\documentclass[11pt]{amsart}
\usepackage{graphicx}
\usepackage{latexsym}
\usepackage{amsfonts,amsmath,amssymb}
\newtheorem{theorem}{Theorem}[section]																
\newtheorem{proposition}[theorem]{Proposition}
\newtheorem{lemma}[theorem]{Lemma}

\newtheorem{corollary}[theorem]{Corollary}

\newcommand{\E}{\operatorname{\mathbb{E}}}
\renewcommand{\P}{\operatorname{\mathbb{P}}}

\usepackage{color}

\def\eps{\varepsilon}

\def\a{\alpha}
\def\be{\beta}

\def\ga{\gamma}
\def\part{\partial}

\def\bA{A^*}
\def\var{\mathrm{var\/}}

\def\bv{\bar{v}}
\def\b1{\bold 1}

\newcommand{\beq}{\begin{equation}}
\newcommand{\eeq}{\end{equation}}

\newtheorem{Theorem}{Theorem}[section]
\newtheorem{Lemma}[Theorem]{Lemma}
\newtheorem{Proposition}[Theorem]{Proposition}

\newcommand{\TT}{\mathcal{T}}

\theoremstyle{remark}

\numberwithin{equation}{section}
\usepackage[paper=a4paper,left=30mm,right=20mm,top=25mm,bottom=30mm]{geometry}
    \newenvironment{dedication}
        {\vspace{6ex}\begin{quotation}\begin{center}\begin{em}}
        {\par\end{em}\end{center}\end{quotation}}

\date{\today}

\begin{document}

\title[Beta-splitting Random Tree]{The Critical Beta-splitting Random Tree: Heights and Related Results} 
\date{}
\author{David Aldous and Boris Pittel}
 \begin{dedication}
\hspace{0cm}
Dedicated to Katy and Monique.
\end{dedication}

\address{Department of Statistics, U.C. Berkeley, 367 Evans Hall, 3860, Berkeley CA 94720}
\email{aldous@stat.berkeley.edu}
\address{Department of Mathematics, Ohio State University, 231 West 18-th Avenue, Columbus OH 43210}
\email{pittel.1@osu.edu}
\begin{abstract} In the critical beta-splitting model of a random $n$-leaf binary tree, leaf-sets are recursively split into subsets, and a set of $m$ leaves is split into subsets containing $i$ and $m-i$ leaves with probabilities proportional to $1/{i(m-i)}$. 
We study the continuous-time model in which the holding time before that split is exponential with rate $h_{m-1}$, the harmonic number.
We (sharply) evaluate the first two moments of the time-height $D_n$ and of the edge-height $L_n$ of a uniform random leaf
(that is, the length of the path from the root to the leaf), and prove the corresponding CLTs.
We study the correlation between the heights of two random leaves of the same tree realization,
and analyze the expected number of splits necessary for a set of $t$ leaves to partially or completely break away from each other.
We give tail bounds for the time-height and the edge-height of the {\em tree}, that is the maximal leaf heights.
We show that there is a limit distribution for the size of a uniform random subtree, and  derive the asymptotics of the mean size.
Our proofs are based on  asymptotic analysis of the attendant (sum-type) recurrences. 
The essential idea is to replace such a recursive equality by a pair of recursive inequalities for which matching asymptotic solutions can be found, allowing one to bound, both ways, the elusive explicit solution of the recursive equality. This reliance on recursive inequalities necessitates usage of Laplace transforms rather than Fourier characteristic functions.
\end{abstract}


\maketitle
\section{Introduction} 

The topic of this paper is a certain random tree model, described below, introduced and discussed  briefly in 1996 in \cite{Ald1} but not subsequently 
studied. 
There is a slight ``applied" motivation as a toy model of phylogenetic trees (see  section \ref{sec:motiv}), 
but our purpose here is to commence  a  theoretical study of the model.
A key point is that it has qualitatively different properties from those of the well-studied random tree models that can be found in (for instance) 
\cite{drmota, janson}.

One fundamental question about a random tree concerns the height of leaves.  
It turns out that this question, and many extensions, can be answered extremely sharply via an analysis of recursions, 
and this paper gives a thorough and detailed account of the range of questions that can be answered by this methodology.

In addressing the Applied Probability community, let us observe that there are also other questions about the model that can be studied via a wide range of 
general modern probability techniques.  Some such work in progress is briefly described in a final
section \ref{sec:3}, and a current overview can be found in the preprint  \cite{Ald2}.

\subsection{The model}

For $m \ge 2$, consider the distribution $(q(m,i),\ 1 \le i \le m-1)$
constructed to be proportional to $\frac{1}{i(m-i)}$.
 Explicitly (by writing $\tfrac{1}{i(m-i)}=\bigl(\tfrac{1}{i}+\tfrac{1}{m-i}\bigr)/m$)
\begin{equation}\label{01}
q(m,i)=\tfrac{m}{2h_{m-1}}\cdot\tfrac{1}{i(m-i)},\,\,1\le i\le m-1,
\end{equation}
where $h_{m-1}$ is the harmonic sum $\sum_{i=1}^{m-1}1/i$. 
Now fix $n \ge 2$.
Consider the process of constructing a random tree by recursively splitting the integer interval 
$[n] = \{1,2,\ldots,n\}$ of ``leaves" as follows.
First specify that there is a left edge and a right edge at the root,
 leading to a left subtree which will have the $L_n$ leaves $\{1,\ldots,L_n\}$
 and a right subtree which will have the $R_n = n - L_n$ leaves $\{L_n + 1,\ldots, n\}$, where $L_n$ 
 (and also $R_n$, by symmetry) has distribution $q(n,\cdot)$. 
 Recursively, a subinterval with $m\ge 2$ leaves is split into two subintervals of random size
  from the distribution $q(m,\cdot)$. 
  Continue until reaching intervals of size $1$, which are the leaves.
This process has a natural tree structure, 
illustrated schematically\footnote{Actual simulations appear in \cite{Ald2}.} in Figure \ref{Fig:1}.
In this discrete-time construction we regard the edges of the tree as having length $1$.
It turns out (see section \ref{sec:motiv})
to be convenient to consider the continuous-time construction in which 
a size-$m$ interval is split at rate $h_{m-1}$, that is after an Exponential$(h_{m-1})$ {\em holding time}.
Once constructed, it is natural to identify ``time" with ``distance": a leaf that appears at time $t$ has
{\em time-height} $t$.
Of course the discrete-time model is implicit within the continuous-time model, and a leaf which appears after 
$\ell$ splits has {\em edge-height} $\ell$.

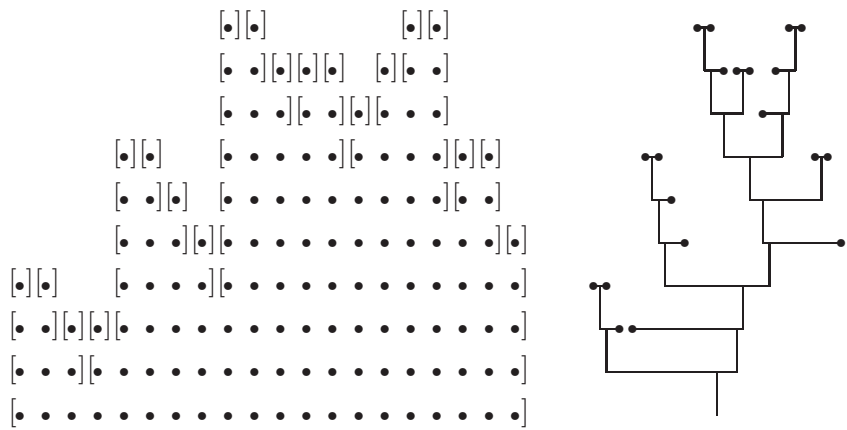
\begin{figure}[ht]
\setlength{\unitlength}{0.045in}
\begin{picture}(100,50)
\multiput(0,0)(3,0){20}{\circle*{0.9}}
\put(-1.3,-0.5){[}
\put(57.7,-0.5){]}
\multiput(0,5)(3,0){20}{\circle*{0.9}}
\put(-1.3,4.5){[}
\put(6.7,4.5){]}
\put(7.7,4.5){[}
\put(57.7,4.5){]}
\multiput(0,10)(3,0){20}{\circle*{0.9}}
\put(-1.3,9.5){[}
\put(3.7,9.5){]}
\put(4.7,9.5){[}
\put(6.7,9.5){]}
\put(7.7,9.5){[}
\put(9.7,9.5){]}
\put(10.7,9.5){[}
\put(57.7,9.5){]}
\multiput(0,15)(3,0){2}{\circle*{0.9}}
\put(-1.3,14.5){[}
\put(0.7,14.5){]}
\put(1.7,14.5){[}
\put(3.7,14.5){]}
\multiput(12,15)(3,0){16}{\circle*{0.9}}
\put(10.7,14.5){[}
\put(21.7,14.5){]}
\put(22.7,14.5){[}
\put(57.7,14.5){]}
\multiput(12,20)(3,0){16}{\circle*{0.9}}
\put(10.7,19.5){[}
\put(18.7,19.5){]}
\put(19.7,19.5){[}
\put(21.7,19.5){]}
\put(22.7,19.5){[}
\put(54.7,19.5){]}
\put(55.7,19.5){[}
\put(57.7,19.5){]}
\multiput(12,25)(3,0){3}{\circle*{0.9}}
\put(10.7,24.5){[}
\put(15.7,24.5){]}
\put(16.7,24.5){[}
\put(18.7,24.5){]}
\multiput(24,25)(3,0){11}{\circle*{0.9}}
\put(22.7,24.5){[}
\put(48.7,24.5){]}
\put(49.7,24.5){[}
\put(54.7,24.5){]}
\multiput(12,30)(3,0){2}{\circle*{0.9}}
\put(10.7,29.5){[}
\put(12.7,29.5){]}
\put(13.7,29.5){[}
\put(15.7,29.5){]}
\multiput(24,30)(3,0){11}{\circle*{0.9}}
\put(22.7,29.5){[}
\put(36.7,29.5){]}
\put(37.7,29.5){[}
\put(48.7,29.5){]}
\put(49.7,29.5){[}
\put(51.7,29.5){]}
\put(52.7,29.5){[}
\put(54.7,29.5){]}
\multiput(24,35)(3,0){9}{\circle*{0.9}}
\put(22.7,34.5){[}
\put(30.7,34.5){]}
\put(31.7,34.5){[}
\put(36.7,34.5){]}
\put(37.7,34.5){[}
\put(39.7,34.5){]}
\put(40.7,34.5){[}
\put(48.7,34.5){]}
\multiput(24,40)(3,0){5}{\circle*{0.9}}
\put(22.7,39.5){[}
\put(27.7,39.5){]}
\put(28.7,39.5){[}
\put(30.7,39.5){]}
\put(31.7,39.5){[}
\put(33.7,39.5){]}
\put(34.7,39.5){[}
\put(36.7,39.5){]}
\multiput(42,40)(3,0){3}{\circle*{0.9}}
\put(40.7,39.5){[}
\put(42.7,39.5){]}
\put(43.7,39.5){[}
\put(48.7,39.5){]}
\multiput(24,45)(3,0){2}{\circle*{0.9}}
\put(22.7,44.5){[}
\put(24.7,44.5){]}
\put(25.7,44.5){[}
\put(27.7,44.5){]}
\multiput(45,45)(3,0){2}{\circle*{0.9}}
\put(43.7,44.5){[}
\put(45.7,44.5){]}
\put(46.7,44.5){[}
\put(48.7,44.5){]}
\put(66,15){\line(1,0){1.5}}
\put(66.75,15){\line(0,-1){5}}
\put(66.75,10){\line(1,0){2.25}}
\put(67.5,10){\line(0,-1){5}}
\put(82.5,10){\line(0,-1){5}}
\put(67.5,5){\line(1,0){15}}
\put(83.25,15){\line(0,-1){5}}
\put(70.5,10){\line(1,0){12.75}}
\put(74.25,20){\line(0,-1){5}}
\put(86.25,20){\line(0,-1){5}}
\put(74.25,15){\line(1,0){12}}
\put(73.5,25){\line(0,-1){5}}
\put(73.5,20){\line(1,0){3}}
\put(85.5,25){\line(0,-1){5}}
\put(85.5,20){\line(1,0){9}}
\put(72.75,30){\line(0,-1){5}}
\put(72.75,25){\line(1,0){2.5}}
\put(72,30){\line(1,0){1.5}}
\put(84,30){\line(0,-1){5}}
\put(91.5,30){\line(1,0){1.5}}
\put(92.25,30){\line(0,-1){5}}
\put(84,25){\line(1,0){8.25}}
\put(81,35){\line(0,-1){5}}
\put(87.75,35){\line(0,-1){5}}
\put(81,30){\line(1,0){6.75}}
\put(79.5,40){\line(0,-1){5}}
\put(83.25,40){\line(0,-1){5}}
\put(79.5,35){\line(1,0){3.75}}
\put(88.5,40){\line(0,-1){5}}
\put(85.5,35){\line(1,0){3}}
\put(78.75,45){\line(0,-1){5}}
\put(78.75,40){\line(1,0){2.25}}
\put(82.5,40){\line(1,0){1.5}}
\put(89.25,45){\line(0,-1){5}}
\put(87,40){\line(1,0){2.25}}
\put(78,45){\line(1,0){1.5}}
\put(88.5,45){\line(1,0){1.5}}
\put(66,15){\circle*{0.9}}
\put(67.5,15){\circle*{0.9}}
\put(69,10){\circle*{0.9}}
\put(70.5,10){\circle*{0.9}}
\put(72,30){\circle*{0.9}}
\put(73.5,30){\circle*{0.9}}
\put(75,25){\circle*{0.9}}
\put(76.5,20){\circle*{0.9}}
\put(78,45){\circle*{0.9}}
\put(79.5,45){\circle*{0.9}}
\put(81,40){\circle*{0.9}}
\put(82.5,40){\circle*{0.9}}
\put(84,40){\circle*{0.9}}
\put(85.5,35){\circle*{0.9}}
\put(87,40){\circle*{0.9}}
\put(88.5,45){\circle*{0.9}}
\put(90,45){\circle*{0.9}}
\put(91.5,30){\circle*{0.9}}
\put(93,30){\circle*{0.9}}
\put(94.5,20){\circle*{0.9}}
\put(80.25,5){\line(0,-1){5}}

\end{picture}
\caption{The discrete time construction for $n = 20$. In the tree, by {\em edges} we mean the $n-1$ vertical edges.  The leaves have edge-heights from $2$ to $9$.}
\label{Fig:1}
\end{figure}

We call the continuous-time model the 
{\em critical beta-splitting random tree}, but must emphasize that the word {\em critical} does not have its usual meaning within branching processes.
Instead, amongst the one-parameter family of splitting probabilities with
\begin{equation}
q(m,i) \propto i^\beta (m-i)^\beta, \ -2 < \beta < \infty \label{beta}
\end{equation}
our parameter value $\beta  = -1$ is {\em critical} in the sense that leaf-heights change from order $n^{-\beta - 1}$ to order $\log n$ 
at that value, as noted many years ago when this family was introduced \cite{Ald1}.

Finally, our results do not use the leaf-labels $\{1,2,\ldots, n \}$ in the interval-splitting construction.
Instead they involve a uniform {\em random} leaf. 
Equivalently, one could take a uniform random permutation of labels and then talk about the leaf with some arbitrary label.

\subsection{Outline of results}
\label{sec:outline}
 Our main focus is on two related random variables associated with the continuous-time random tree on $n$
leaves:
\begin{itemize}
\item $D_n = $ time-height of a uniform random leaf;
\item $L_n = $ edge-height of a uniform random leaf.
\end{itemize}
We start with sharp asymptotic formulas for the moments of $D_n$ and $L_n$. They are of considerable interest in their own right, and also because the techniques are then extended for analysis of the limiting distributions, with the moments estimates enabling us to guess what those distributions should be.   

Write $\zeta(\cdot)$  for the Riemann zeta-function, $\zeta(r):=\sum_{j=1}^{\infty} \tfrac{1}{j^r}$, ($r>1$).
Note that $\zeta(2) = \pi^2/6$ and that $\zeta^{-1}(2)$ below means $1/\zeta(2)$, not the inverse function.
Write $\ga$ for the Euler-Masceroni constant, which will appear frequently in our analysis: $\sum_{j=1}^n\tfrac{1}{j}=\log n +\ga+O(n^{-1})$.
Asymptotics are as $n \to \infty$.
\begin{theorem}\label{thmA}
\begin{align*}
\E[D_n]&=\zeta^{-1}(2)\log n+O(1),\\
\var(D_n)&=(1+o(1))\tfrac{2\zeta(3)}{\zeta^3(2)}\log n,
\end{align*}
and, contingent on a numerically supported ``$h$-ansatz'' (see section \ref{sec:ansatz}),
\[
\E[D_n] = \zeta^{-1}(2) \log n+c_0-\tfrac{1}{2 \zeta(2)} n^{-1} +O(n^{-2})
\]
for a constant $c_0$ estimated numerically, and
\[
\var(D_n)=\tfrac{2\zeta(3)}{\zeta^3(2)}\log n +O(1).
\]
\end{theorem}
\begin{theorem}\label{thmB}
\begin{align*}
\E[L_n]&=\tfrac{1}{2\zeta(2)}\log^2 n+\tfrac{\ga\,\zeta(2)+\zeta(3)}{\zeta^2(2)}\log n +O(1),\\
\var(L_n)&=\tfrac{2\zeta(3)}{3\zeta^3(2)}\log^3 n+O(1).
\end{align*}
\end{theorem}
\noindent
The various parts of Theorem \ref{thmA} are proved in sections \ref{sec:Dmoments}--\ref{sec:var} and \ref{sec:Dnorm},
and Theorem \ref{thmB} is proved in section \ref{sec:Lmoments}.
These theorems immediately yield the WLLNs (weak laws of large numbers) for $D_n$ and $L_n$, with rates, as follows.
\begin{corollary}\label{thmC} In probability
\[
\P\Bigl(\big|\tfrac{D_n}{\E[D_n]}-1\big|\ge \eps\Bigr),\,\,\P\Bigl(\big|\tfrac{L_n}{\E[L _n]}-1\big|\ge \eps\Bigr)=O\bigl(\eps^{-2} \log^{-1} n\bigr).
\]
\end{corollary}
\noindent
Consider next the
 time-height  $\mathcal D_n$ and the edge-height $\mathcal L_n$ of the random tree itself, 
 that is the largest time length and the largest edge length  of a path from the root to a leaf.
By upper-bounding the Laplace transforms of $D_n$ and $L_n$,
we prove  in sections \ref{sec:Rtail}   and  \ref{sec:Dtail} 
\begin{theorem}\label{thmD} There exists $\rho>0$ such that for all $\eps\in (0,1)$ we have 
\[
\P\Bigl(\mathcal D_n\ge (2+\eps)\log n\Bigr)\le \tfrac{1}{n^{\rho\eps}},
\]
\end{theorem}
\begin{theorem}\label{thmE} 
Let $\be=\min_{\a>1/\log 2}\bigl[\a+\tfrac{4\a^2\zeta(3)}{\a\log 2-1}\bigr]\approx 42.9$. For $\eps\in (0,1)$, 
\[
\P\Bigl(\mathcal L_n\ge (1+\eps)\be\log^2 n)\Bigr)\le \exp\bigl(-\Theta(\eps\log n)\bigr).
\]
\end{theorem}
\noindent
We conjecture that both $\tfrac{\mathcal D_n}{\log n}$ and $\tfrac{\mathcal L_n}{\log^2 n}$ converge, in probability, 
to constants.

 The definitions of $D_n$ and $L_n$ involve two levels of randomness, the random tree and the random leaf within the tree.
 To study the interaction between levels, it is natural to consider the correlation between the heights of two leaves within the same realization of the random tree.
 Write $D_n^{(1)}$ and $D_n^{(2)}$ for the time-heights of two distinct leaves chosen uniformly from all  pairs of leaves.
 We study the correlation coefficient defined by
\[
r_n=\tfrac{\E[D_n^{(1)} D_n^{(2)}]-\E^2[D_n]}{\text{Var}(D_n)},
\]
and prove in section \ref{sec:corr} 
\begin{theorem}\label{thmF} Contingent on the $h$-ansatz, $r_n=O(\log^{-1}n)$, that is asymptotically $D_n^{(1)}$ and $D_n^{(2)}$ are uncorrelated.
 \end{theorem}
\noindent We conjecture that a similar result holds for the correlation coefficient of $L_n^{(1)}$ and $L_n^{(2)}$, the 
edge-heights of two distinct, uniformly random leaves, {\it independently\/} of the $h$-ansatz.

Returning to properties of $D_n$ and $L_n$, in sections \ref{sec:Dnorm}  and  \ref{sec:Lnorm} we will prove  the CLTs corresponding to the means and variances in Theorems \ref{thmA} and \ref{thmB}.
\begin{theorem}\label{thmG} In distribution, and with all their moments, 
\[
\frac{D_n-\zeta^{-1}(2)\log n}{\sqrt{\tfrac{2\zeta(3)}{\zeta^3(2)}\log n}},\quad \frac{L_n-(2\zeta(2))^{-1}\log^2 n}{\sqrt{\tfrac{2\zeta(3)}{3\zeta^3(2)}\log^3 n}}\Longrightarrow 
\mathrm{Normal}(0,1) .
\]
\end{theorem}
\noindent
The sharp asymptotic estimates of the moments of $D_n$ and $L_n$, and the ample numeric evidence in the case of $D_n$, provided a compelling evidence that both $D_n$ and $L_n$ must be asymptotically normal. However, the proof of Theorem \ref{thmG} does not use these estimates, providing instead an alternative verification of the {\it leading\/} terms in those estimates,  without relying on the $h$-ansatz.

After posting the original preprint version of this article, alternative proofs of these CLTs have appeared in preprints. 
Via a martingale CLT \cite{Ald2} (continuous model);
via the contraction method \cite{kolesnik}   (discrete model);
and via the theory of regenerative composition structures \cite{iksanovCLT} (discrete model).
 Presumably these methods can also be applied to the alternate model.
Of course, in Theorem \ref{thmG} there is presumably {\em joint convergence} to a bivariate Gaussian limit. 
It would be interesting to see which method would be best for proving such joint convergence.

Like Theorems \ref{thmD} and \ref{thmE}, the proof of Theorem \ref{thmG} is based on showing convergence of the Laplace transform  for the (properly centered and scaled) leaf height to that of $\mathrm{Normal}(0,1)$. Why Laplace, but not Fourier?  Because, even though there is enough independence to optimistically expect asymptotic normality, our variables 
are too far from being the sums of essentially independent terms. So, the best we could do is to use recurrences to bound
the (real-valued) Laplace transforms recursively both ways, by those of the Normals, whose parameters we choose to satisfy, asymptotically, the respective recursive inequalities. The added feature here is that we get convergence of the moments as well.  

Leaving Laplace versus Fourier issue aside, there are many cases when a limited moment information and the recursive nature of the process can be used to establish asymptotic normality, but the standard techniques hardly apply, see \cite{DFP}, \cite{MP}, \cite{Pit1}, \cite{Pit2}, \cite{Pit}, \cite{Pit3}. The concrete details vary substantially, of course. For instance,
in \cite{Pit} it was shown that the total number of linear extensions of the random, tree-induced, partial order is lognormal, by showing convergence of all semi-invariants, rather than of the Laplace transforms. In \cite{Pit3}, for the proof of a two-dimensional CLT for the number of vertices and arcs in the giant strong component of the random digraph, boundedness of the Fourier transform made it indispensable.
The  unifying feature of these diverse arguments is the recurrence equation for the chosen transform.

\begin{figure}[ht]
\setlength{\unitlength}{0.045in}
\begin{picture}(100,50)(7,0)
\multiput(0,0)(3,0){20}{\circle*{0.9}}
\put(-1.3,-0.5){[}
\put(57.7,-0.5){]}
\multiput(0,5)(3,0){20}{\circle*{0.9}}
\put(-1.3,4.5){[}
\put(6.7,4.5){]}
\put(7.7,4.5){[}
\put(57.7,4.5){]}
\multiput(0,10)(3,0){20}{\circle*{0.9}}
\put(-1.3,9.5){[}
\put(3.7,9.5){]}
\put(4.7,9.5){[}
\put(6.7,9.5){]}
\put(7.7,9.5){[}
\put(9.7,9.5){]}
\put(10.7,9.5){[}
\put(57.7,9.5){]}
\put(0,15){\circle*{0.9}}
\put(3,15){\circle{0.9}}
\put(-1.3,14.5){[}
\put(0.7,14.5){]}
\put(1.7,14.5){[}
\put(3.7,14.5){]}
\multiput(12,15)(3,0){16}{\circle*{0.9}}
\put(10.7,14.5){[}
\put(21.7,14.5){]}
\put(22.7,14.5){[}
\put(57.7,14.5){]}
\multiput(12,20)(3,0){16}{\circle*{0.9}}
\put(10.7,19.5){[}
\put(18.7,19.5){]}
\put(19.7,19.5){[}
\put(21.7,19.5){]}
\put(22.7,19.5){[}
\put(54.7,19.5){]}
\put(55.7,19.5){[}
\put(57.7,19.5){]}
\multiput(12,25)(3,0){2}{\circle*{0.9}}
\put(18,25){\circle{0.9}}
\put(10.7,24.5){[}
\put(15.7,24.5){]}
\put(16.7,24.5){[}
\put(18.7,24.5){]}
\multiput(24,25)(3,0){11}{\circle*{0.9}}
\put(22.7,24.5){[}
\put(48.7,24.5){]}
\put(49.7,24.5){[}
\put(54.7,24.5){]}
\put(12,30){\circle{0.9}}
\put(15,30){\circle*{0.9}}
\put(10.7,29.5){[}
\put(12.7,29.5){]}
\put(13.7,29.5){[}
\put(15.7,29.5){]}
\multiput(24,30)(3,0){10}{\circle*{0.9}}
\put(54,30){\circle{0.9}}
\put(22.7,29.5){[}
\put(36.7,29.5){]}
\put(37.7,29.5){[}
\put(48.7,29.5){]}
\put(49.7,29.5){[}
\put(51.7,29.5){]}
\put(52.7,29.5){[}
\put(54.7,29.5){]}
\multiput(24,35)(3,0){9}{\circle*{0.9}}
\put(22.7,34.5){[}
\put(30.7,34.5){]}
\put(31.7,34.5){[}
\put(36.7,34.5){]}
\put(37.7,34.5){[}
\put(39.7,34.5){]}
\put(40.7,34.5){[}
\put(48.7,34.5){]}
\multiput(24,40)(3,0){2}{\circle*{0.9}}
\put(30,40){\circle{0.9}}
\multiput(33,40)(3,0){2}{\circle*{0.9}}
\put(22.7,39.5){[}
\put(27.7,39.5){]}
\put(28.7,39.5){[}
\put(30.7,39.5){]}
\put(31.7,39.5){[}
\put(33.7,39.5){]}
\put(34.7,39.5){[}
\put(36.7,39.5){]}
\multiput(42,40)(3,0){3}{\circle*{0.9}}
\put(40.7,39.5){[}
\put(42.7,39.5){]}
\put(43.7,39.5){[}
\put(48.7,39.5){]}
\multiput(24,45)(3,0){2}{\circle*{0.9}}
\put(22.7,44.5){[}
\put(24.7,44.5){]}
\put(25.7,44.5){[}
\put(27.7,44.5){]}
\put(45,45){\circle{0.9}}
\put(48,45){\circle*{0.9}}
\put(43.7,44.5){[}
\put(45.7,44.5){]}
\put(46.7,44.5){[}
\put(48.7,44.5){]}
\put(66,15){\line(1,0){1.5}}
\put(66.75,15){\line(0,-1){5}}
\put(66.75,10){\line(1,0){2.25}}
\put(67.5,10){\line(0,-1){5}}

\put(82.5,10){\line(0,-1){5}}
\put(67.5,5){\line(1,0){15}}
\put(83.25,15){\line(0,-1){5}}
\put(70.5,10){\line(1,0){12.75}}
\put(74.25,20){\line(0,-1){5}}
\put(86.25,20){\line(0,-1){5}}
\put(74.25,15){\line(1,0){12}}
\put(73.5,25){\line(0,-1){5}}
\put(73.5,20){\line(1,0){3}}
\put(85.5,25){\line(0,-1){5}}
\put(85.5,20){\line(1,0){9}}
\put(72.75,30){\line(0,-1){5}}
\put(72.75,25){\line(1,0){2.5}}
\put(72,30){\line(1,0){1.5}}
\put(84,30){\line(0,-1){5}}
\put(91.5,30){\line(1,0){1.5}}
\put(92.25,30){\line(0,-1){5}}
\put(84,25){\line(1,0){8.25}}
\put(81,35){\line(0,-1){5}}
\put(87.75,35){\line(0,-1){5}}
\put(81,30){\line(1,0){6.75}}
\put(79.5,40){\line(0,-1){5}}
\put(83.25,40){\line(0,-1){5}}
\put(79.5,35){\line(1,0){3.75}}
\put(88.5,40){\line(0,-1){5}}
\put(85.5,35){\line(1,0){3}}
\put(78.75,45){\line(0,-1){5}}
\put(78.75,40){\line(1,0){2.25}}
\put(82.5,40){\line(1,0){1.5}}
\put(89.25,45){\line(0,-1){5}}
\put(87,40){\line(1,0){2.25}}
\put(78,45){\line(1,0){1.5}}
\put(88.5,45){\line(1,0){1.5}}
\put(66,15){\circle*{0.9}}
\put(67.5,15){\circle{0.9}}

\put(69,10){\circle*{0.9}}
\put(70.5,10){\circle*{0.9}}
\put(72,30){\circle{0.9}}
\put(73.5,30){\circle*{0.9}}
\put(75,25){\circle{0.9}}

\put(76.5,20){\circle*{0.9}}
\put(78,45){\circle*{0.9}}
\put(79.5,45){\circle*{0.9}}
\put(81,40){\circle{0.9}}
\put(82.5,40){\circle*{0.9}}
\put(84,40){\circle*{0.9}}

\put(85.5,35){\circle*{0.9}}
\put(87,40){\circle*{0.9}}
\put(88.5,45){\circle{0.9}}

\put(90,45){\circle*{0.9}}
\put(91.5,30){\circle*{0.9}}

\put(93,30){\circle{0.9}}
\put(94.5,20){\circle*{0.9}}
\put(80.25,5){\line(0,-1){5}}


\put(110.25,5){\line(0,-1){5}}
\put(97.5,5){\line(1,0){15}}
\put(97.5,5){\circle{0.9}}
\put(112.5,10){\line(0,-1){5}}
\put(113.25,15){\line(0,-1){5}}
\put(112.5,10){\line(1,0){0.75}}

\put(104.25,15){\line(1,0){12}}
\put(104.25,20){\line(0,-1){5}}
\put(103.5,25){\line(0,-1){5}}
\put(104.25,20){\line(-1,0){0.75}}
\put(101.75,25){\line(1,0){3.5}}
\put(105,25){\circle{0.9}}
\put(101.5,25){\circle{0.9}}

\put(116.25,20){\line(0,-1){5}}
\put(115.5,25){\line(0,-1){5}}
\put(116.25,20){\line(-1,0){0.75}}
\put(114,25){\line(1,0){8.25}}
\put(122.5,25){\circle{0.9}}
\put(114,30){\line(0,-1){5}}
\put(111,30){\line(1,0){6.75}}
\put(111,30){\circle{0.9}}
\put(117.75,30){\circle{0.9}}

\end{picture}
\caption{The left and center diagrams show $t = 6$ leaves $\circ$ in the $n = 20$-leaf tree in Figure \ref{Fig:1}.
The right diagram is the pruned spanning tree on those leaves, with $8$ edges.}
\label{Fig:1b}
\end{figure}
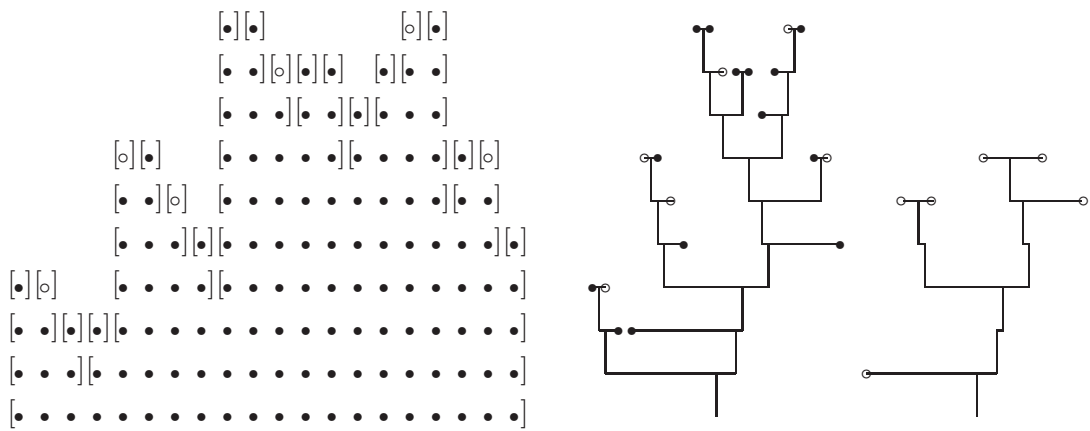

The structure theory studied in \cite{Ald2} involves the notion of {\em pruned spanning tree}, 
illustrated in Figure \ref{Fig:1b},
and here we study its edge-length.
Given a set $T$ of $t:= |T|  < n$ leaves of the tree on $n$ leaves, there is spanning tree on those leaves and the root; the edges of the spanning tree are the union of the 
edges on the paths to these leaves. 
Now we can ``prune" this spanning tree by cutting the end segment of each path back to the internal vertex $v$ where  it branches from the other paths;
the spanning tree on those branchpoints $v$ forms the pruned spanning tree.
Equivalently, the edges of the pruned spanning tree are the edges in the paths from the root to vertices $v$
such that {\it each\/} of the two subtrees rooted at $v$'s children has at least one leaf from $T$.
Write $S^*_{n,t}$ for the number of such edges, when $T$ is a uniform random choice of $t<n$ leaves. 
In section \ref{sec:pruned} we prove
\begin{theorem}\label{thmH} With $B(t_1,t_2)=\tfrac{\Gamma(t_1)\Gamma(t_2)}{\Gamma(t)}$, we have
\begin{equation*}
\E[S^*_{n,t}]=\a(t)\log n+O(1),\,\,\a(t)=\biggl(h_{t-1}-\sum_{t_1+t_2=t}B(t_1,t_2)\biggr)^{-1},
\end{equation*}
\end{theorem}
\noindent
along with a related result (Proposition \ref{thm11}) for the edge-height of the first branch-point in the pruned tree.
At long last, the Riemann zeta-function has suddenly loosened its grip, and appropriately the Beta-function has taken the stage. 

Finally in section \ref{sec:countingsubtrees} we prove
\begin{theorem}\label{thmt>2.0} 
Let $\mathbf{u}_n:=\{u_n(t)\}_{t\le n}$ be the distribution of the number of leaves in a subtree rooted at a uniform random vertex, that is one of the 
$2n-1$ leaves or branchpoints.
 The sequence $\{\mathbf{u}_n\}_{n\ge 1}$ converges
to a proper distribution $\mathbf{u}$. However 
$\sum_{t\ge 1}t u_n(t)\sim\tfrac{3}{2\pi^2}\log^2 n$.
\end{theorem}

\subsection{Motivation and background} 
\label{sec:motiv}

The one-parameter family at (\ref{beta}) was introduced in  \cite{Ald1} in 1996 as a toy model for phylogenetic trees.
It was observed in \cite{me_yule} that, in splits $m \to (i,m-i)$ in real-world phylogenetic trees, 
the median size of the smaller subtree scaled roughly as $m^{1/2}$.
This data is not consistent with more classical random tree models, where the median size would be 
$O(\log m)$ or $\Theta(m)$.
However, in the ``critical" case $\beta  = -1$ of the model studied in this paper, 
that median size is indeed order $m^{1/2}$, because 
\[
2 \sum_{i < m^{1/2}} q(m,i) \to \tfrac{1}{2} \mbox{ as } m \to \infty .
\]
Of course the model is not a biologically meaningful ``forwards in time" model, but is
a mathematically basic model that could be used for comparison with more realistic models \cite{lambert} that are deliberately constructed 
in the mathematical biology literature to reproduce features of real phylogenetic trees.
The distributional results of this model might therefore be useful for some future ``applied" project of that kind.

\section{The proofs}
 Let $\tau_{\nu}$ be the holding time before a split  of a subset of size $\nu$. So $\tau_{\nu}$ has Exponential distribution with rate $h_{\nu-1}$. By the definition of the splitting
process, for $\nu\ge 2$ we have: with $q(\nu,i)=\tfrac{\nu}{2h_{\nu-1}}\,\tfrac{1}{i(\nu-i)}$ as at \eqref{01},
\[
D_{\nu}=\left\{\begin{aligned}
&\tau_{\nu}+D_i,&&\text{with probability }q(\nu,i)\tfrac{i}{\nu},\,\,i=1,\dots,\nu-1,\\
&\tau_{\nu}+D_{\nu-i},&&\text{with probability }q(\nu,i)\tfrac{\nu-i}{\nu},\,\,i=1,\dots,\nu-1. \end{aligned}\right.
\]
Introduce $\phi_{\nu}(u)=\E[e^{u D_{\nu}}]$, the Laplace transform of the distribution of $D_{\nu}$; so, $\phi_1(u)=1$. 
The equation above implies that for $\nu\ge 2$,
\begin{multline}\label{2}
\phi_{\nu}(u)\!=\!\sum_{k=1}^{\nu-1}\!q(\nu,k)\Bigl(\!\tfrac{k}{\nu}\E[\exp(u(\tau_{\nu}+D_k))]+\tfrac{\nu-k}{\nu}\E[\exp(u(\tau_{\nu}+D_{\nu-k}))]\!\Bigr)\\
=2 \E \bigl[\exp(u\tau_{\nu})\bigr]\sum_{k=1}^{\nu-1}\tfrac{k}{\nu}\phi_k(u)\,q_{\nu,k}
=\tfrac{1}{h_{\nu-1}-u}\sum_{k=1}^{\nu-1}\tfrac{\phi_k(u)}{\nu-k}.
\end{multline}
Furthermore, introduce $f_{\nu}(u)=\E[e^{u L_{\nu}}]$, the Laplace transform of the distribution of $L_{\nu}$; so $f_1(u)=1$. In this case we have: for $\nu\ge 2$,
\[
L_{\nu}=\left\{\begin{aligned}
&1+L_i,&&\text{with probability }q(\nu,i)\tfrac{i}{\nu},\,\,i=1,\dots,\nu-1,\\
&1+L_{\nu-i},&&\text{with probability }q(\nu,i)\tfrac{\nu-i}{\nu},\,\,i=1,\dots,\nu-1.\end{aligned}\right.
\]
Therefore
\begin{multline}\label{2.2}
f_{\nu}(u)\!=\!\sum_{k=1}^{\nu-1}\!q(\nu,k)\Bigl(\!\tfrac{k}{\nu}\E[\exp(u(1+L_k))]+\tfrac{\nu-k}{\nu}\E[\exp(u(1+L_{\nu-k}))]\!\Bigr)\\
=2e^u\sum_{k=1}^{\nu-1}\tfrac{k}{\nu}f_k(u)\,q_{\nu,k}
=\tfrac{e^u}{h_{\nu-1}}\sum_{k=1}^{\nu-1}\tfrac{f_k(u)}{\nu-k}.
\end{multline}

In particular, we make extensive use of the following
 fundamental recurrence  for $\E[D_{\nu}]$: 
\begin{equation}\label{2.01}
\E[D_{\nu}]=\tfrac{1}{h_{\nu-1}}\biggl(1+\sum_{k=1}^{\nu-1}\tfrac{\E[D_k]}{\nu-k}\biggr).                            
\end{equation}
This follows directly from the hold-jump construction of the random tree, or by differentiating both sides of \eqref{2} at $u=0$.

\subsection{The moments of $D_n$.}
\label{sec:Dmoments}
Our first result includes one part of Theorem \ref{thmA}.
\begin{proposition}\label{1.5}
\begin{eqnarray*}
&&\zeta^{-1}(2)\log n\le \E[D_n]\le \max\{0,1+\log(n-1)\},\,\,n\ge 2,\\
&&\qquad\qquad\quad\E[D_n]=\zeta^{-1}(2)\log n +O(1).
\end{eqnarray*}
\end{proposition}
\begin{proof} 
The proof has three steps.

{\bf (i)\/} Let us prove that $\E[D_n]\ge \tfrac{6}{\pi^2}\log n$. Introduce $\theta_n=A\log n$. Then $\E[D_1]=0=\theta_1$. If we find $A$ such that
\begin{equation}\label{new2}
\theta_n\le  \frac{1}{h_{n-1}}\biggl(1+\sum_{k=1}^{n-1}\tfrac{\theta_k}{n-k}\biggr),\quad n\ge 2,
\end{equation}
then, by induction on $n$, $\E[D_n]\ge \theta_n$ for all $n\ge 1$. We compute
\begin{multline*}
 \frac{1}{h_{n-1}}\biggl(1+\sum_{k=1}^{n-1}\tfrac{\theta_k}{n-k}\biggr)=\frac{1}{h_{n-1}}\biggl(1+\sum_{k=1}^{n-1}\tfrac{A\log k}{n-k}\biggr)\\
 =\frac{1}{h_{n-1}}\biggl(1+A(\log n) h_{n-1}+A\sum_{k=1}^{n-1}\tfrac{\log(k/n)}{n(1-k/n)}\biggr)\\
=\theta_n+\tfrac{1}{h_{n-1}}\biggl(1+A\sum_{k=1}^{n-1}\tfrac{\log(k/n)}{n(1-k/n)}\biggr)\\
\ge \theta_n+\tfrac{1}{h_{n-1}}\biggl(1-A\int_0^1\tfrac{\log(1/x)}{1-x}\,dx\biggr).
\end{multline*}
The inequality holds since the integrand is positive and decreasing.  Since  
\[
\int_0^1\tfrac{\log(1/x)}{1-x}\,dx=\sum_{j\ge 0}\int_0^1x^j \log(1/x)\,dx=\sum_{j\ge 0}\tfrac{1}{(j+1)^2}=\zeta(2)=\tfrac{\pi^2}{6},
\]
we deduce that \eqref{new2} holds if we select $A=\tfrac{6}{\pi^2}=\zeta^{-1}(2)$.

\noindent{\bf Note.\/} The proof above is the harbinger of things to come, including the next part. The seemingly naive idea is to replace a recurrence equality by a recurrence {\it inequality\/} for which an exact solution can be found and then to use it to upper bound the otherwise-unattainable solution of the recurrence {\it equality\/}. Needless to say, it is critically important to have a good guess as to how that ``hidden'' solution behaves asymptotically.

{\bf (ii)\/} Let us prove that 
$\E[D_n]\le f(n):= \max\{0,1+\log(n-1)\}$ for $n\ge 2$. 
This is true for $n=1, 2$ since $\E[D_1]=0$, $\E[D_2]=1$.
Notice that $1+\log(x-1)\le x-1$ for $x\in (1,2]$. So $f(x)\le g(x)$, $\forall\,x>1$,
where $g(x)=x-1$ for $x\in [1,2]$, $g(x)=1+\log(x-1)$ for $x\ge 2$, {\it\/} and $g(x)$ is concave for $x\ge 1$.
So, similarly to \eqref{new2}, it is enough to show that  $g(n)$ satisfies  
\begin{equation}\label{new2.01}
g(n)\ge \frac{1}{h_{n-1}}\biggl(1+\sum_{i=1}^{n-1}\tfrac{g(i)}{n-i}\biggr),\quad n\ge 2.
\end{equation}
By concavity of $g(x)$ for $x\ge 1$, we have
\begin{multline*}
 \frac{1}{h_{n-1}}\biggl(\,1+\sum_{i=1}^{n-1}\tfrac{g(i)}{n-i}\biggr)\le \tfrac{1}{h_{n-1}}+g\biggl(\sum_{i=1}^{n-1}\tfrac{i}{n-i}\biggr)\\
 =\tfrac{1}{h_{n-1}}+g\bigl(n-\tfrac{n-1}{h_{n-1}}\bigr)\le \tfrac{1}{h_{n-1}}+g(n)-g'(n)\bigl(\tfrac{n-1}{h_{n-1}}\bigr),\\
 \end{multline*}
 which is exactly $g(n)$, since $g'(n)=\tfrac{1}{n-1}$ for $n>1$.

 {\bf (iii)\/}
Write $\E[D_n]=\tfrac{6}{\pi^2}\log n + u_n$, so that $u_n\ge 0$ and $u_1=0$. Let us prove that $u_n=O(1)$.  Using \eqref{2.01} we have
\begin{multline}\label{2.02}
u_n=\tfrac{1}{h_{n-1}}\biggl(1+\sum_{k=1}^{n-1}\tfrac{u_k}{n-k}\biggr)+\tfrac{6}{\pi^2}\biggl(\tfrac{1}{h_{n-1}}\sum_{k=1}^{n-1}\tfrac{\log k}{n-k}-\log n\biggr)\\
=\tfrac{1}{h_{n-1}}\biggl(1+\sum_{k=1}^{n-1}\tfrac{u_k}{n-k}\biggr)+\tfrac{6}{\pi^2 h_{n-1}}\sum_{k=1}^{n-1}\tfrac{\log(k/n)}{n-k}.
\end{multline} 
\noindent
The proof of {\bf (iii)\/} depends on the following rather sharp asymptotic formula for the last sum, which we believe to be new.
We defer the proof of the lemma.
\begin{Lemma}
\label{L:nk}
\[\sum_{k=1}^{n-1}\tfrac{\log(k/n)}{n-k}=-\zeta(2) +\tfrac{\log(2\pi e)}{2n}+\tfrac{\log n}{12 n^2} +O(n^{-2}). \]
\end{Lemma}
\noindent
Granted this estimate, the recurrence \eqref{2.02} becomes
\begin{equation}\label{2.035}
\begin{aligned}
u_n&=\tfrac{\zeta^{-1}(2)}{h_{n-1}}\bigl(\tfrac{\log(2\pi en)}{2n}+\tfrac{\log n}{12 n^2}+O(n^{-2})\bigr)\\
&\quad\quad+\tfrac{1}{h_{n-1}}\sum_{k=1}^{n-1}\tfrac{u_k}{n-k},\,\,n\ge 2,\,\, u_1=0.
\end{aligned}
\end{equation}
It is  easy to check that the sequence $x_n:=\tfrac{n-1}{n}$ satisfies the recurrence
\[
x_n=\tfrac{1}{n}+\tfrac{1}{h_{n-1}}\sum_{k=1}^{n-1}\tfrac{x_k}{n-k},\quad n\ge 2,\quad x_1=0.
\]
As the explicit term on the RHS of \eqref{2.035} is asymptotic to $\tfrac{\zeta^{-1}(2)}{2n}$, we can deduce that $u_n=O(1)$, establishing  {\bf (iii)\/}.
Indeed, by the triangle inequality, the equation \eqref{2.035} implies that 
\[
|u_n|\le \tfrac{c}{n}+\tfrac{1}{h_{n-1}}\sum_{k=1}^{n-1}\tfrac{|u_k|}{n-k}.
\]
By induction on $n$, this inequality coupled with the recurrence for $x_n$ imply that $|u_n|\le 2c x_n\le 2c$.
\end{proof}

\noindent
{\em Proof of Lemma \ref{L:nk}.}

First, we have: for $n\ge 2$ 
\begin{multline}\label{2.03}
\sum_{k=1}^{n-1}\tfrac{\log(k/n)}{n-k}=\sum_{k=1}^{n-1}\biggl(\tfrac{\log(k/n)}{n}+\tfrac{k\log(k/n)}{n(n-k)}\biggr)
=\tfrac{1}{n}\log\tfrac{(n-1)!}{n^{n-1}}+\sum_{k=1}^{n-1}\tfrac{(k/n)\log(k/n)}{n-k}.\\
\end{multline}
By Euler's summation formula (Graham, Knuth, and Patashnik \cite{GKP}, (9.78)), if $f(x)$ is a smooth differentiable function
for $x\in [a,b]$ such that the even derivatives are all of the same sign, then for every $m\ge 1$
\begin{multline}\label{2.031}
\sum_{a\le k<b}\!f(k)\!=\!\int_a^b\!f(x)\,dx-\tfrac{1}{2}f(x)\Big|_{a}^{b}
+\sum_{\ell=1}^m 
\tfrac{B_{2\ell}}{(2\ell)!} f^{(2\ell-1)}(x)\big|_{a}^b+\theta_m\tfrac{B_{2m+2}}{(2m+2)!} f^{(2m+1)}(x)\big|_{a}^b.\\
\end{multline}
Here $\theta_m\in (0,1)$ and $\{B_{2\ell}\}$ are even Bernoulli numbers, defined by $\tfrac{z}{e^z-1}=\sum_{\mu\ge 0}B_{\mu}\tfrac{z^{\mu}}{\mu!}$. 
The equation \eqref{2.031} was used in \cite{GKP} to show that 
\begin{multline*}%
\sum_{1\le k<n}\log k=n\log n-n+\tfrac{1}{2}\log\tfrac{2\pi}{n}
+\sum_{\ell=1}^m\tfrac{B_{2\ell}}{2\ell(2\ell-1)n^{2\ell-1}}+\theta_{m,n}\tfrac{B_{2m+2}}{(2m+2)(2m+1)n^{2m+1}},\\
\end{multline*}
$\theta_{m,n}\in (0,1)$. Here $f(x)=\log x$, so that $f^{(2\ell)}(x)< 0$ for $x\ge 1$ and $\ell\ge 1$. Using this estimate for $m=1$,
we obtain a sharp version of Stirling's formula:
\begin{equation}\label{2.032}
\tfrac{1}{n}\log\tfrac{(n-1)!}{n^{n-1}}=-1+\tfrac{\log(2\pi n)}{2n}+O(n^{-2}).
\end{equation}

Consider the sum in the bottom RHS of \eqref{2.03}.
This time, take $f(x)=\tfrac{(x/n)\log(x/n)}{n-x}$, $x\in [1,n-1]$, and $f(n):=-\tfrac{1}{n}$.
 Let us show that $f^{(2\ell)}(x)>0$ for $x\in (0,n)$, or
equivalently that  $g^{(2\ell)}(y)>0$ for $y\in (0,1)$, where $g(y):=\tfrac{y \log y}{1-y}$. We have
\begin{multline*}
g(y)=-\log y+\tfrac{\log y}{1-y}=-\log y-\sum_{j\ge 1}\tfrac{(1-y)^{j-1}}{j}
=-\log(1-z)-\sum_{j\ge 1}\tfrac{z^{j-1}}{j}, \quad z:=1-y.\\
\end{multline*}
So, we need to show that
\[
\bigl(-\log(1-z)\bigr)^{(2\ell)}\ge \biggl(\sum_{j\ge 1}\tfrac{z^{j-1}}{j}\biggr)^{(2\ell)},
\]
or equivalently that 
\[
\tfrac{(2\ell-1)!}{(1-z)^{2\ell}}\ge \sum_{j>2\ell}\tfrac{(j-1)_{2\ell}\,z^{j-1-2\ell}}{j}.
\]
This inequality will follow if we prove a stronger inequality\footnote{$[z^\nu]$ denotes the coefficient of $z^\nu$.}, namely that for every $\nu\ge 0$
\[
[z^{\nu}]\tfrac{(2\ell-1)!}{(1-z)^{2\ell}}\ge [z^{\nu}] \sum_{j>2\ell}\tfrac{(j-1)_{2\ell}\,z^{j-1-2\ell}}{j}.
\]
But this is equivalent to
\[
(2\ell+\nu-1)!\ge \tfrac{(2\ell+\nu)!}{2\ell+\nu+1},
\]
which is obviously true. Therefore, applying \eqref{2.031}, we have: with $\theta'_{m,n}\in (0,1)$,
\begin{multline}\label{2.033}
\sum_{k=1}^{n-1}\tfrac{(k/n)\log(k/n)}{n-k}=\int_{1/n}^1g(y)\,dy-\tfrac{1}{2n} g(y)\big|_{1/n}^1\\
+\sum_{\ell=1}^m\tfrac{B_{2\ell}}{n^{2\ell}(2\ell)!}\,  g^{(2\ell-1)}(y)\big|_{1/n}^1+\theta'_{m,n}\tfrac{B_{2m+2}}{n^{2m+2}(2m+2)!}\, g^{(2m+1)}(y)\big|_{1/n}^1.
\end{multline}
For the first terms in \eqref{2.033}
\begin{align*}
\int_{1/n}^1g(y)\,dy&=\int_0^1\tfrac{y\log y}{1-y}\,dy-\sum_{j\ge 1}\int_0^{1/n}y^j\log y\,dy\\
&=-\zeta(2)+1+(\log n)\sum_{j\ge 2}n^{-j}j^{-1}+\sum_{j\ge 2}n^{-j} j^{-2};\\
g(y)\big|_{1/n}^1&=-1+\tfrac{\log n}{n-1}.\\
\end{align*}
The integrals were evaluated using the more general identities \eqref{zzz} and \eqref{2.495} later.

For the next term in  \eqref{2.033} we need $g^{(2\ell-1)}(y)\big|_{1/n}^1$. 
 We use the Newton-Leibniz formula and evaluate 
$g^{(2\ell-1)}(1/n)$ and $g^{(2\ell-1)}(1)$ using respectively
\begin{align*}
g^{(2\ell-1)}(y)&=\sum_{j=0}^{2\ell-1}\binom{2\ell-1}{j}(\log y)^{(j)}\bigl(\tfrac {y}{1-y}\bigr)^{(2\ell-1-j)},\\
g^{(2\ell-1)}(y)&=\sum_{j=0}^{2\ell-1}\binom{2\ell-1}{j}y^{(j)}\biggl(\tfrac{\log y}{1-y}\biggr)^{(2\ell-1-j)} .
\end{align*}
In the second sum there are only two non-zero terms, for $j=0$ and $j=1$, and using $\tfrac{\log y}{1-y}=-\sum_{j\ge 1}\tfrac{(1-y)^{j-1}}{j}$ we obtain, with some work, that
\[
g^{(2\ell-1)}(1)=-\tfrac{(2\ell-2)!}{2\ell}.
\]
For $g^{(2\ell-1)}(1/n)$, we use $\bigl(\tfrac {y}{1-y}\bigr)^{(\mu)}=\bigl(\tfrac{1}{1-y}\bigr)^{(\mu)}$ for $\mu>0$, and after some more protracted work we obtain 
\[
g^{(2\ell-1)}(1/n)=-(\log n)\tfrac{(2\ell-1)!}{(1-n^{-1})^{2\ell}}+
\sum_{j=1}^{2\ell-2}n^j\cdot\tfrac{(2\ell-1)_j}{j(1-n^{-1})^{2\ell-j}}+n^{2\ell-2}\cdot\tfrac{(2\ell-2)!}{1-1/n}.
\]
Therefore
\begin{align*}
g^{(2\ell-1)}(y)\big|_{1/n}^1&=-\tfrac{(2\ell-2)!}{2\ell}+(\log n)\tfrac{(2\ell-1)!}{(1-n^{-1})^{2\ell}}\\
&\quad-\sum_{j=1}^{2\ell-2}n^j\cdot\tfrac{(2\ell-1)_j}{j(1-n^{-1})^{2\ell-j}}-n^{2\ell-2}\cdot \tfrac{(2\ell-2)!}{1-1/n}.
\end{align*}
This term enters the RHS of \eqref{2.033} with the factor $n^{-2\ell}$, making the product of order $n^{-2}$ regardless of $m\ge 1$. And the remainder term in \eqref{2.033} is of order $n^{-2}$, again independently of $m\ge 1$. So we choose
the simplest $m=1$. Collecting all the pieces we transform \eqref{2.033} into 
\begin{equation}\label{2.034}
\sum_{k=1}^{n-1}\tfrac{(k/n)\log(k/n)}{n-k}=-\zeta(2)+1+\tfrac{1}{2n}+\tfrac{\log n}{12n^2}+O(n^{-2}).
\end{equation}
\noindent
So, combining \eqref{2.03}, \eqref{2.032}, and \eqref{2.034}, we have
\[
\sum_{k=1}^{n-1}\tfrac{\log(k/n)}{n-k}=-\zeta(2)+\tfrac{\log(2\pi e)}{2n}+\tfrac{\log n}{12 n^2} +O(n^{-2})
\]
which is the assertion  of Lemma \ref{L:nk}.

This completes the proof of Proposition \ref{1.5}.

\subsection{An ansatz for sharper results}
\label{sec:ansatz}
Knowing that  $\E[D_n]=\zeta^{-1}(2)\log n+O(1)$, it seems natural to seek more refined estimates  by imagining that
\[ \E[D_n] = \zeta^{-1}(2)\log n +\sum_{j\ge 0}c_j n^{-j} \]
{\it almost \/} satisfies the recurrence, and then calculating $c_j$.
Let us call this the {\em h-ansatz}, being analogous to a known expansion for $h_n$.
So to use this ansatz we write
\[ w_n:=\sum_{j\ge 0}c_j n^{-j} \]
and seek to identify the $c_j$ from the recurrence \eqref{2.035}, which we re-write as follows.
\begin{equation}\label{2.515}
\begin{aligned}
w_n&=\tfrac{d_1\log n}{nh_{n-1}}+\tfrac{d_2}{nh_{n-1}}+\tfrac{1}{h_{n-1}}\sum_{k=2}^{n-1}\tfrac{w_k}{n-k},\quad n\ge 2,\\
d_1&=\tfrac{\zeta^{-1}(2)}{2},\quad d_2=\tfrac{\zeta^{-1}(2)}{2}\log(2\pi e).
\end{aligned}
\end{equation} 
Here 
\[
\tfrac{\log n}{h_{n-1}}=1-\tfrac{\ga}{\log n}+O(\log^{-2}n),
\]
where 
\begin{equation}\label{2.5155}
\gamma:=1-\sum_{j=2}^{\infty}\tfrac{\zeta(j)-1}{j}\approx 0.5772156649,
\end{equation}
is the Euler-Masceroni constant coming from $h_{\nu}=\log\nu+\ga+O(\nu^{-1})$, \cite{GKP}. For $n\ge 3$, using $\tfrac{1}{k(n-k)}=n^{-1}\bigl(\tfrac{1}{k}+\tfrac{1}{n-k}\bigr)$, we have
\begin{align*}
\sum_{k=2}^{n-1}\tfrac{w_k}{n-k}&=\sum_{j\ge 0}c_j\sum_{k=2}^{n-1}\tfrac{1}{k^j(n-k)}\\
&=c_0\bigl(h_{n-1}-\tfrac{1}{n-1}\bigr)+c_1n^{-1}\bigl(2h_{n-1}-\tfrac{n}{n-1}\bigr)\\
&\quad+n^{-1}\sum_{j\ge 2}c_j\sum_{k=2}^{n-1}\bigl(\tfrac{1}{k^j}+\tfrac{1}{k^{j-1}(n-k)}\bigr)\\
&=c_0\bigl(h_{n-1}-\tfrac{1}{n-1}\bigr)+c_1n^{-1}\bigl(2h_{n-1}-\tfrac{n}{n-1}\bigr)\\
&\quad+n^{-1}\sum_{j\ge 2}c_j(\zeta(j)-1) +O(n^{-2}\log n).
\end{align*}
Therefore
\begin{multline*}
\tfrac{d_1\log n}{nh_{n-1}}+\tfrac{d_2}{nh_{n-1}}+\tfrac{1}{h_{n-1}}\sum_{k=2}^{n-1}\tfrac{w_k}{n-k}-w_n\\
=\tfrac{d_1+c_1}{n} +\tfrac{1}{nh_{n-1}}\biggl(-d_1\ga+d_2-c_0-c_1+\sum_{j\ge 2}c_j\bigl(\zeta(j)-1\bigr)\biggr)
+O(n^{-2}).\\
\end{multline*}
So, selecting
\begin{equation}\label{2.516}
c_1=-d_1=-\tfrac{3}{\pi^2},\quad c_0=d_2+\sum_{j\ge 2}\bigl(c_j+\tfrac{d_1}{j}\bigr)(\zeta(j)-1),
\end{equation}
(as suggested by \eqref{2.5155}) we have
\begin{equation*}
\tfrac{d_1\log n}{nh_{n-1}}+\tfrac{d_2}{nh_{n-1}}+\tfrac{1}{h_{n-1}}\sum_{k=2}^{n-1}\tfrac{w_k}{n-k}-w_n=O(n^{-2}).
\end{equation*}
Therefore $w_n=\sum_{j\ge 0}c_j n^{-j}$ satisfies \eqref{2.035} within the additive error $O(n^{-2})$, provided that $\{c_j\}_{j\ge 0}$ satisfies \eqref{2.516}.
It is worth noticing that $c_0$ is well defined for every $\{c_j\}_{j\ge 2}$ provided that the series in \eqref{2.516} converges. 
The constant $c_0$ can be viewed as a counterpart of the Euler-Masceroni constant $\ga$.
Strikingly, $c_0$ depends on {\it all\/} $c_j$, $j\ge 2$, while $c_1$ is determined uniquely from the requirement that
$w_n$ satisfies \eqref{2.515} within $O(n^{-2})$ error. 

So the conclusion is
\begin{Proposition}
\label{P:sharp}
Assuming the h-ansatz, there exists a constant $c_0$ such that 
\begin{equation}\label{2.517}
\E[D_n] =\tfrac{6}{\pi^2}\log n+c_0-\tfrac{3}{\pi^2} n^{-1} +O(n^{-2}).
\end{equation}
\end{Proposition}
This is another part of Theorem \ref{thmA}.
One can calculate $\E [D_n]$ numerically via the basic recurrence, and doing so up to $n = 400,000$ gives a good fit\footnote{Taking the coefficient of $n^{-1}$ as unknown, 
the fit to this data is   $0.30408$, compared to $ \tfrac{3}{\pi^2}  = 0.30396$. } 
to \eqref{2.517} 
with $c_0 = 0.7951556604....$.
We do not have a conjecture for the explicit value of $c_0$.

In what follows,  we will use only a weak corollary of \eqref{2.517}, namely
\begin{equation}\label{2.518}
\E[D_n] =\tfrac{6}{\pi^2}\log n+c_0 +O(n^{-1}).
\end{equation}
 Paradoxically, the actual value of $c_0$ will be immaterial as well.

\subsection{The recursion for variance}
Parallel to the recursion  \eqref{2.01}  for expectations, here is the recursion for variance.
\begin{lemma}\label{lem0.9} Setting $v_n:=\var(D_n)$, we have
\begin{equation}\label{new3.01}
v_n=\tfrac{1}{h_{n-1}}\sum_{k=1}^{n-1}\tfrac{v_k+(\E[D_n]-\E[D_k])^2}{n-k}.
\end{equation}
\end{lemma}
\begin{proof} 
Differentiating twice both sides of \eqref{2} at $u=0$, we get 
\begin{eqnarray*}
\E[D_n^2]&=&\tfrac{2}{h_{n-1}^3}\cdot h_{n-1}+\tfrac{2}{h_{n-1}^2}\sum_{k=1}^{n-1}\tfrac{\E[D_k]}{n-k}+
\tfrac{1}{h_{n-1}}\sum_{k=1}^{n-1}\tfrac{\E[D_k^2]}{n-i}\\
&=&\tfrac{2}{h_{n-1}^2}\biggl(1+\sum_{k=1}^{n-1}\tfrac{\E[D_k]}{n-k}\biggr)+\tfrac{1}{h_{n-1}}\sum_{k=1}^{n-1}\tfrac{\E[D_k^2]}{n-k}\\
&=&\tfrac{2\E[D_n]}{h_{n-1}}+\tfrac{1}{h_{n-1}}\sum_{k=1}^{n-1}\tfrac{\E[D_k^2]}{n-k}.
\end{eqnarray*}
Since $v_n=\E[D_n^2]-\E^2[D_n]$, the equation above becomes
\[
v_n=\tfrac{2\E[D_n]}{h_{n-1}}+\tfrac{1}{h_{n-1}}\sum_{k=1}^{n-1}\tfrac{v_k+\E^2[D_k]}{n-k}-\E^2[D_n] .
\]
The identity \eqref{new3.01} holds because, by \eqref{2.01},
\begin{equation*}
\tfrac{2\E[D_n]}{h_{n-1}}+\tfrac{1}{h_{n-1}}\sum_{k=1}^{n-1}\tfrac{\E^2[D_k]}{n-k}-\E^2[D_n]
=\tfrac{1}{h_{n-1}}\sum_{k=1}^{n-1}\tfrac{(\E[D_n]-\E[D_k])^2}{n-k}.
\end{equation*}
\end{proof}
\noindent {\bf Note.\/} The equation \eqref{new3} could be obtained by using the ``law of total variance". We preferred the above derivation as more direct in the present context, inconceivable without Laplace transform. Besides, the similar argument will be used later to derive  a recurrence for variance of the edge length of the random path. It will be almost the ``same'' as \eqref{new3.01}, but with an unexpected, if not shocking, additive term $-1$ on the RHS.

\subsection{Sharp estimates of $\var (D_n)$}
\label{sec:var}
Assuming the h-ansatz, and using \eqref{new3.01}, we are able to obtain the following sharp estimate, asserted as part of Theorem \ref{thmA}.
\begin{proposition}\label{thm0.91} Contingent on the h-ansatz,
\begin{equation*}
v_n=\tfrac{2\zeta(3)}{\zeta^3(2)}\log n +O(1).\quad n\ge 2.
\end{equation*}
\end{proposition}
{\bf Note.\/} It is the term $(\E[D_n]-\E[D_k])^2$ in \eqref{new3.01} that necessitates our reliance on the h-ansatz. 
Comfortingly, the first-order result $\var (D_n) \sim \tfrac{2\zeta(3)}{\zeta^3(2)}\log n$ follows from the CLT proof in section \ref{sec:Dnorm}, independently of the h-ansatz.
\begin{proof} 
By \eqref{2.518}, we have
\begin{eqnarray}
\bigl(\E[D_n]-\E[D_k]\bigr)^2 &=& \zeta^{-2}(2) \bigl(\log(n/k) +O(k^{-1})\bigr)^2\nonumber \\
&=&\zeta^{-2}(2)\Bigl(\log^2(n/k)+O\bigl(k^{-1}\log(n/k)\bigr)+O(k^{-2})\Bigr). \label{3.03}
\end{eqnarray}
We need the estimates
\begin{align*}
\sum_{k=1}^{n-1}\tfrac{\log(n/k)}{k(n-k)}&=n^{-1}\sum_{k=1}^{n-1}(k^{-1}+(n-k)^{-1})\log(n/k)=O(n^{-1}\log^2 n),\\
\sum_{k=1}^{n-1}\tfrac{1}{k^2(n-k)}&=n^{-1}\sum_{k=1}^{n-1}\bigl(k^{-2}+n^{-1}(k^{-1}+(n-k)^{-1})\bigr)=O(n^{-1}).
\end{align*}
Consider the dominant term in \eqref{3.03}.  Observe that the function $\tfrac{\log^2(n/x)}{n-x}$ is convex.
So, using \eqref{2.031} for $m=0$, we obtain
\begin{eqnarray*}
\sum_{k=1}^{n-1}\tfrac{\log^2(n/k)}{n-k} &=& \int_1^n\tfrac{\log^2(n/x)}{n-x}\,dx+O(n^{-1}\log^2 n)\nonumber\\
&=& \int_0^1\tfrac{\log^2(1/x)}{1-x}\,dx+O(n^{-1}\log^2 n)\nonumber \\
&=& 2\zeta(3)+O(n^{-1}\log^2 n). \label{2.0315}
\end{eqnarray*}
To explain the final equality, by induction on $r$ and  integrating by parts, we obtain
\begin{equation}
 \int_0^1 z^j \log^r z\,dz=(-1)^r\tfrac{r!}{(j+1)^{r+1}}. 
 \label{zzz}
 \end{equation}
Consequently
\begin{equation}
\label{2.495}
\int_0^1\tfrac{\log^r z}{1-z}\,dz=\int_0^1(\log^r z)\sum_{j\ge 0}z^j\,dz=(-1)^r r!\,\zeta(r+1),\quad r\ge 1
\end{equation}
used for $r=2$ at \eqref{2.0315}.
Now the recursion in Lemma \ref{lem0.9} becomes
\[
v_n=\tfrac{1}{h_{n-1}}\biggl(\tfrac{2\zeta(3)}{\zeta^2(2)} + O(n^{-1}\log^2 n)+\sum_{k=1}^{n-1}\tfrac{v_k}{n-k}\biggr).
\]
Recalling that
\[
\E[D_n]=\tfrac{1}{h_{n-1}}\biggl(1+\sum_{k=1}^{n-1}\tfrac{\E[D_k]}{n-k}\biggr),
\]
it follows that $w_n:=\Big|v_n-\tfrac{2\zeta (3)}{\zeta^2(2)} \E[D_n]\Big|$ satisfies
\begin{equation}\label{3.039}
w_n\le \tfrac{1}{h_{n-1}}\biggl( cn^{-1}\log ^2 n+\sum_{k=1}^{n-1}\tfrac{w_k}{n-k}\biggr),\quad n\ge 2,\,\, w_1=0,
\end{equation}
for some constant $c>0$. 
Let us prove that the sequence 
\[ z_n:=c\bigl(\log^2(14)-\tfrac{\log^2(14n)}{n}\bigr) \]
satisfies 
\begin{equation}\label{3.04}
z_n\ge \tfrac{1}{h_{n-1}}\biggl( cn^{-1}\log ^2 n+\sum_{k=1}^{n-1}\tfrac{z_k}{n-k}\biggr), \quad n\ge 2.
\end{equation}
Because $z_1=0=w_1$, we will get then, predictably by induction using \eqref{3.039}, that $w_n\le z_n$. 
Let us prove \eqref{3.04}. For $g(x):=-\tfrac{\log^2(14x)}{x}$, we have 
\begin{align*}
g'(x)&=x^{-2}\bigl(\log^2(14x)-2\log(14x)\bigr),\\
g^{\prime\prime}(x)&=-\tfrac{2}{x^3}\bigl[\log^2(14x)-3\log(14x)+1\bigr] <0,\quad x\ge 1,
\end{align*}
because $\log(14)>2.63>\tfrac{3+\sqrt{5}}{2}$, the larger of two roots of $x^2-3x+1$. Therefore $g(x)$ is {\it concave\/} on $[1,\infty)$.
So,
\begin{eqnarray*}
\tfrac{1}{h_{n-1}}\sum_{k=1}^{n-1}\tfrac{g(k)}{n-k}
\le g\biggl(\tfrac{1}{h_{n-1}}\sum_{k=1}^{n-1}\tfrac{k}{n-k}\biggr)&=&g\bigl(n-\tfrac{n-1}{h_{n-1}}\bigr) \\
&\le &g(n)-g'(n)\tfrac{n-1}{h_{n-1}}\\
&=&g(n)-n^{-2}\bigl(\log^2(14n)-2\log(14 n)\bigr)\tfrac{n-1}{h_{n-1}}.
\end{eqnarray*}
Since $z_k=c\bigl(\log^2(14)+g(k)\bigr)$, we obtain then
\begin{multline*}
\tfrac{1}{h_{n-1}}\biggl( \tfrac{c\log ^2 n}{n}+\sum_{k=1}^{n-1}\tfrac{z_k}{n-k}\biggr)\\
\le z_n+\tfrac{c}{h_{n-1}}\bigl[\tfrac{\log^2 n}{n}-n^{-2}(n-1)\bigl(\log^2(14n)-2\log(14 n)\bigr)\bigr]<z_n,
\end{multline*}
because the expression within square brackets is easily shown to be negative for $n\ge 2$.  
This establishes \eqref{3.04}.
\end{proof}

\subsection{How correlated are  leaf-heights?}
\label{sec:corr}
Recall the statement of Theorem \ref{thmF}, copied below as Theorem \ref{thmF2}.
 To study the interaction between the two levels of randomness, it is natural to consider the correlation between leaf heights.
 Write $D_n^{(1)}$ and $D_n^{(2)}$ for the time-heights, within the same realization of the random tree,
 of two distinct leaves chosen uniformly over pairs of leaves. Both time-heights individually are distributed as $D_n$,
 the time height of the uniformly random leaf.
 We study the correlation coefficient defined by
\[
r_n=\tfrac{\E[D_n^{(1)} D_n^{(2)}]-\E^2[D_n]}{\text{Var}(D_n)}.
\]
\begin{theorem}\label{thmF2} Contingent on the $h$-ansatz, $r_n=O(\log^{-1}n)$.
\end{theorem}
\begin{proof} 
Recall the splitting distribution $n \to (L_n,R_n)$ at \eqref{01}:
\begin{equation}
 \P(L_n = i) = q(n,i) = \frac{n}{2h_{n-1}} \frac{1}{i(n-i)} = q(n,n-i) , \ 1 \le i \le n-1 .
 \label{rule2}
 \end{equation}
There is a natural recursion for $Z_\nu := D_\nu^{(1)} \cdot D_\nu^{(2)}$, as follows.
\begin{equation}\label{3.049}
Z_{\nu}\overset{d}=\left\{\begin{aligned}
&(\tau_{\nu}+D_i^{(1)})(\tau_{\nu}+D_i^{(2)}),&&\text{with probability }q(\nu,i)\cdot\tfrac{(i)_2}{(\nu)_2},\\
&(\tau_{\nu}+D_{\nu-i}^{(1)})(\tau_{\nu}+D_{\nu-i}^{(2)}),&&\text{with probability }q(\nu,i)\cdot\tfrac{(\nu-i)_2}{(\nu)_2},\\
&(\tau_{\nu}+D_{i}^{(1)})(\tau_{\nu}+D_{\nu-i}^{(2)}),&&\text{with probability }q(\nu,i)\cdot\tfrac{i(\nu-i)}{(\nu)_2},\\
&(\tau_{\nu}+D_{i}^{(2)})(\tau_{\nu}+D_{\nu-i}^{(1)}),&&\text{with probability }q(\nu,i)\cdot\tfrac{i(\nu-i)}{(\nu)_2}.
\end{aligned}\right.
\end{equation}
Here $\tau_{\nu}$ is the Exponential$(h_{\nu -1})$ hold time.
The first two cases correspond to the two leaves being in the same subtree, so their heights are dependent, whereas 
the last two cases correspond to the two leaves being in the different subtrees, so their heights are (conditionally) independent.

Consequently
\begin{multline*}
\E[Z_{\nu}|\, L_\nu =i]=\Bigl(\tfrac{2}{h^2_{\nu-1}}+\tfrac{2}{h_{\nu-1}}\E[D_i]+\E[Z_i]\Bigr)\tfrac{(i)_2}{(\nu)_2}\\
+\Bigl(\tfrac{2}{h^2_{\nu-1}}+\tfrac{2}{h_{\nu-1}}\E[D_{\nu-i}]+\E[Z_{\nu-i}]\Bigr)\tfrac{(\nu-i)_2}{(\nu)_2}\\
+2\Bigl(\tfrac{2}{h^2_{\nu-1}}+\tfrac{1}{h_{\nu-1}}\bigl(\E[D_i]+\E[D_{\nu-i}]\bigr) +\E[D_i] \cdot\E[D_{\nu-i}]\Bigr)\tfrac{i(\nu-i)}{(\nu)_2},
\end{multline*}
or, with a bit of algebra,
\begin{multline*}
\E[Z_{\nu}|\, L_\nu=i]=\tfrac{2}{h^2_{\nu-1}}+\tfrac{2i\E[D_i]}{\nu h_{\nu-1}}+\tfrac{2(\nu-i)\E[D_{\nu-i}]}{\nu h_{\nu-1}}\\
+\tfrac{1}{(\nu)_2}\Bigl((i)_2\E[Z_i]+(\nu-i)_2\E[Z_{\nu-i}]+2i(\nu-i)\E[D_i]\,\E[D_{\nu-i}]\Bigr).
\end{multline*}
Using \eqref{rule2} we obtain then
\begin{multline*}
\E[Z_{\nu}]=\sum_{i=1}^{\nu-1} q(\nu,i)\E[Z_{\nu}|\, L_\nu =i]=\tfrac{2}{h^2_{\nu-1}}+\tfrac{2}{h_{\nu-1}^2}\sum_{i=1}^{\nu-1}
\tfrac{\E[D_i]}{\nu-i}\\
+\tfrac{1}{(\nu-1)h_{\nu-1}}\sum_{i=1}^{\nu-1}\E[D_i]\, \E[D_{\nu-i}]+\tfrac{1}{(\nu-1)h_{\nu-1}}\sum_{i=1}^{\nu-1}\tfrac{(i-1)\E[Z_i]}{\nu-i}.
\end{multline*}
So, using $\E[D_{\nu}]=\tfrac{1}{h_{\nu-1}}\Bigl(1+\sum_{i=1}^{\nu-1}\tfrac{\E[D_i]}{\nu-i}\Bigr)$, we arrive at
\begin{equation}\label{3.05}
\begin{aligned}
\E[Z_{\nu}]&=\tfrac{1}{(\nu-1)h_{\nu-1}}\sum_{i=1}^{\nu-1}\tfrac{(i-1)\E[Z_i]}{\nu-i}\\
&+\tfrac{2\E[D_{\nu}]}{h_{\nu-1}} +\tfrac{1}{(\nu-1)h_{\nu-1}}\sum_{i=1}^{\nu-1}\E[D_i] \,\E[D_{\nu-i}].
\end{aligned}
\end{equation}
We use \eqref{3.05} to sharply estimate $\E[Z_{\nu}]$ and then estimate $r_n=\frac{\E[Z_{\nu}]-\E^2[D_n]}{\text{Var}(D_n)}$. To start, 
\[
\tfrac{2\E[D_{\nu}]}{h_{\nu-1}}=2\zeta^{-1}(2)+O(\log^{-1}\nu).
\]
Secondly,
\begin{multline*}
\E[D_i]\,\E[D_{\nu-i}]=\bigl[\zeta^{-1}(2)\log i +c_0 +O(i^{-1})\bigr]\\
\times\bigl[\zeta^{-1}(2)\log (\nu-i)+c_0 +O((\nu-i)^{-1})\bigr].\\
\end{multline*}
The leading contribution to $\sum_i \E[D_i]\,\E[D_{\nu-i}]$ comes from
\begin{multline*}
\zeta^{-2}(2)\sum_{i=1}^{\nu-1}\log i\cdot \log (\nu-i)\\
=\zeta^{-2}(2)(\nu-1)\log^2\nu+2\zeta^{-2}(2)\log \nu\sum_{i=1}^{\nu-1}\log(i/\nu)\\
+\zeta^{-2}(2)\sum_{i=1}^{\nu-1}\log(i/\nu)\log ((\nu-i)/\nu)\\
=\zeta^{-2}(2)\nu\log^2\nu+2\zeta^{-2}(2)\nu\log \nu\int_0^1\log x\,dx+O(\nu)\\
=\zeta^{-2}(2)\big(\nu\log^2\nu-2\nu\log\nu\bigr)+O(\nu).
\end{multline*}
The secondary contribution to $\sum_i \E[D_i]\,\E[D_{\nu-i}]$ comes from $c_0\zeta^{-1}(2)(\log i+\log(\nu-i))$, and it equals $2c_0\zeta^{-1}(2)\nu\log \nu+O(\nu)$. The terms $c_0, O(i^{-1}), O((\nu-i)^{-1})$ contribute jointly $O(\nu)$. Altogether,
\[
\sum_{i=1}^{\nu-1}\E[D_i]\,\E[D_{\nu-i}]=\zeta^{-2}(2)\big(\nu\log^2\nu-2\nu\log\nu\bigr)+2c_0\zeta^{-1}(2)\nu\log \nu+O(\nu).
\]
Therefore the equation \eqref{3.05} becomes
\begin{multline}\label{3.06}
\E[Z_{\nu}]=\tfrac{1}{(\nu-1)h_{\nu-1}}\sum_{i=1}^{\nu-1}\tfrac{(i-1)\E[Z_i]}{\nu-i}+2\zeta^{-1}(2)\\
+\zeta^{-2}(2)\big(\log\nu-2-\ga\bigr)+2c_0\zeta^{-1}(2)+O(\log^{-1}\nu).
\end{multline}
Let us look at an approximate solution $\widetilde{E}(\nu):=A \log^2\nu+B\log\nu$. The RHS of the above equation is
\begin{multline*}
\tfrac{1}{(\nu-1)h_{\nu-1}}\sum_{i=1}^{\nu-1}\tfrac{(i-1)(A\log^2 i+B\log i)}{\nu-i}+2\zeta^{-1}(2)
+\zeta^{-2}(2)\big(\log\nu-2-\ga\bigr)+2c_0\zeta^{-1}(2)+O(\log^{-1}\nu).\\
\end{multline*}
Here, since $\sum_i\tfrac{i-1}{\nu-i}=(\nu-1)(h_{\nu-1}-1)$, we have
\begin{multline*}
\tfrac{1}{(\nu-1)h_{\nu-1}}\sum_{i=1}^{\nu-1}\tfrac{(i-1)\log^2 i}{\nu-i}=\tfrac{1}{(\nu-1)h_{\nu-1}}\sum_{i=1}^{\nu-1}\tfrac{(i-1)\bigl(\log(i/\nu)+\log\nu\bigr)^2}{\nu-i}\\
=\tfrac{h_{\nu-1}-1}{h_{\nu-1}}\log^2\nu+\tfrac{2\log\nu}{(\nu-1)h_{\nu-1}}\sum_{i=1}^{\nu-1}\tfrac{(i-1)\log(i/\nu)}{\nu-i}+
\tfrac{1}{(\nu-1)h_{\nu-1}}\sum_{i=1}^{\nu-1}\tfrac{(i-1)\log^2(i/\nu)}{\nu-i}\\
=\log^2\nu-\log\nu+\ga+2\int_0^1\tfrac{x\log x}{1-x}\,dx+O\bigl(\log^{-1}\nu\bigr)\\
=\log^2\nu-\log\nu+\ga+2(1-\zeta(2))+O\bigl(\log^{-1}\nu\bigr),
\end{multline*}
and
\[
\tfrac{1}{(\nu-1)h_{\nu-1}}\sum_{i=1}^{\nu-1}\tfrac{(i-1)\log i}{\nu-i}=\log\nu-1+O(\log^{-1}\nu).
\]
Therefore, with $\widetilde{E}(\cdot)$ instead of $\E[Z_{\cdot}]$, the RHS of the equation \eqref{3.06} becomes
\begin{multline*}
A\bigl(\log^2\nu-\log \nu+\ga+2(1-\zeta(2))\bigr)+B\bigl(\log \nu-1)\\
+\zeta^{-2}(2)\big(\log\nu-2-\ga\bigr)+2(c_0+1)\zeta^{-1}(2)+O(\log^{-1}\nu).
\end{multline*}
And we need this to be equal to $\widetilde{E}(\nu):=A\log^2\nu+B\log\nu$ within an additive error $O(\log^{-1}\nu)$, meaning that 
\begin{align*}
&\qquad\qquad\qquad-A+B+\zeta^{-2}(2)=B,\\
&A\bigl[\ga+2(1-\zeta(2))\bigr]-B-(2+\ga)\zeta^{-2}(2)+2(c_0+1)\zeta^{-1}(2)=0,
\end{align*}
or explicitly
\begin{equation}\label{3.061}
A=\zeta^{-2}(2), \quad B=2c_0\zeta^{-1}(2).
\end{equation}
With these $A$ and $B$, our approximation $\widetilde{E}(\nu)$ satisfies the same equation \eqref{3.06} as $\E[Z_{\nu}]$, excluding an exact value
of the remainder term $O(\log^{-1}\nu)$, of course.  Consequently, $\Delta(\nu):=\big|\E[Z_{\nu}]-\widetilde{E}(\nu)\big|$ satisfies
\begin{equation}\label{3.07}
\Delta(\nu)\le\tfrac{1}{(\nu-1)h_{\nu-1}}\sum_{i=1}^{\nu-1}\tfrac{(i-1)\Delta(i)}{\nu-i}+O(\log^{-1}\nu),\quad \Delta(1)=0.
\end{equation}
With $\mathcal U_{\nu}:=(\nu-1)\Delta(\nu)$, the resulting equation is a special case of the later equation \eqref{4.0531} with the remainder term $O(\nu^{t-1}\log^{-1}\nu)$, when $t=2$. Applying the bound for the solution proved there, we obtain that $\mathcal U_{\nu}=O(\nu)$, or that $\Delta(\nu)=O(1)$. Thus
\[
\E[Z_{\nu}]=A\log^2\nu+B\log\nu+O(1).
\]
Combining this formula  with \eqref{3.061}, $r_n=\tfrac{\E[Z_{\nu}]-\E^2[D_n]}{\text{Var}(D_n)}$ and $\E[D_n]=\zeta^{-1}(2)\log n+c_0+O(n^{-1})$, we compute
\begin{multline*}
r_n= \frac{\zeta^{-2}(2)\log^2n+2c_0\zeta^{-1}(2)\log n-\bigl(\zeta^{-1}(2)\log n+c_0\bigr)^2+O(1)}{\tfrac{2\zeta(3)}{\zeta^3(2)}\log n +O(1)}=O(\log^{-1}n).\\
\end{multline*}
\end{proof}

\noindent{\bf Note.\/} We do not need the h-ansatz in the rest of the paper.

\subsection{Bounding the time-height of the random tree.}
\label{sec:Rtail}

Consider now  the time-height  $\mathcal D_n$  of the random tree itself, 
 that is the maximum leaf time-height.
 We re-state Theorem \ref{thmD}, together with a tail bound on $D_n$. 

\begin{proposition}\label{0.915} {\bf (i)\/} For some $\rho>0$ and all $\eps\in (0,\pi^2/6-1)$, 
\[
\P\Bigl(D_n\ge \tfrac{6}{\pi^2}(1+\eps)\log n\Bigr)=O(n^{-\rho\eps}).
\]
{\bf (ii)\/} For some $\rho^\prime$ and all $\eps\in (0,1)$, 
\[
\P\Bigl(\mathcal D_n\ge 2(1+\eps)\log n\Bigr)=O(n^{-\rho^\prime\eps}).
\]
\end{proposition}

 \begin{proof} {\bf (i)\/} Since the tree with $\nu$ leaves has $\nu - 1$ non-leaf vertices, rather crudely $D_{\nu}$ is stochastically dominated by the sum of $\nu-1$ 
independent exponentials with rate $1$. Therefore, for $u<1$, the Laplace transform $\phi_{\nu}(u):=\E[e^{u D_{\nu}}]$ is bounded above by $(1-u)^{-\nu}$. 
Recall \eqref{2}: 
 \begin{equation*}
\phi_{\nu}(u)=\tfrac{1}{h_{\nu-1}-u}\sum_{k=1}^{\nu-1}\tfrac{\phi_k(u)}{\nu-k}, \quad \nu\ge 2.
\end{equation*}
Pick $\eps'<\eps$ and introduce $\a= \tfrac{6}{\pi^2}(1+\eps')$ (so $\alpha <1$) and $\psi_{\nu}(u)=\exp\bigl(u\a\log \nu\bigr)$. Let us prove that 
\begin{equation}\label{3.051}
\psi_{\nu}(u)\ge\tfrac{1}{h_{\nu-1}-u}\sum_{k=1}^{\nu-1}\tfrac{\psi_k(u)}{\nu-k},
\end{equation}
 if $u\in (0,1)$ is sufficiently small, and $\nu>1$ sufficiently large.

First note that
\[
\psi_k(u)=\psi_{\nu}(u)\exp\bigl(u\a\log(k/\nu)\bigr),\quad k\le\nu.
\]
Therefore
\begin{multline*}
\tfrac{1}{\psi_{\nu}(u)(h_{\nu-1}-u)}\sum_{k=1}^{\nu-1}\tfrac{\psi_k(u)}{\nu-k}=\tfrac{1}{h_{\nu-1}-u}\sum_{k=1}^{\nu-1}
\tfrac{\exp\bigl(u\a\log(k/\nu)\bigr)}{\nu-k}\\
=\bigl(1-\tfrac{u}{h_{\nu-1}}\bigr)^{-1}\cdot\Bigl(1+\tfrac{1}{h_{\nu-1}}
\sum_{k=1}^{\nu-1}\tfrac{\exp\bigl(u\a\log(k/\nu)\bigr)-1}{\nu-k}\Bigr)\\
=\biggl(1+\tfrac{u}{h_{\nu-1}}+O\bigl(\tfrac{u^2}{h_{\nu-1}^2}\bigr)\biggr)\\
\times\biggl[1+\tfrac{u}{h_{\nu-1}}\!\sum_{k=1}^{\nu-1}\tfrac{\a\log(k/\nu)}{\nu-k}+O\biggl(\tfrac{u^2}{h_{\nu-1}}\sum_{k=1}^{\nu-1}
\tfrac{\log^2(k/\nu)}{\nu-k}\biggr)\biggr];
\end{multline*}
(where we used $|e^x-1-x|\le x^2/2$, for $x\le 0$).  
So, since $\a=\zeta^{-1}(2)(1+\eps')$,
\begin{eqnarray}\label{3.0515}
\tfrac{1}{\psi_{\nu}(u)(h_{\nu-1}-u)}\sum_{k=1}^{\nu-1}\tfrac{\psi_k(u)}{\nu-k}
&=&1+\tfrac{u}{h_{\nu-1}}\Bigl(1+\a\sum_{k=1}^{\nu-1}\tfrac{\log(k/\nu)}{\nu-k}\Bigr)+O\bigl(\tfrac{u^2}{h_{\nu-1}}\bigr)\\
&\le& 1+\tfrac{u}{h_{\nu-1}}\Bigl(1+\a\bigl(-\zeta(2)+\tfrac{\log(\nu e)}{\nu-1}\bigr)+O\bigl(\tfrac{u^2}{h_{\nu-1}}\bigr)\Bigr)\nonumber\\
&=&1-\tfrac{u}{h_{\nu-1}}\bigl(\eps'-\zeta^{-1}(2)(1+\eps)\tfrac{\log(\nu e)}{\nu-1}\bigr)+O\bigl(\tfrac{u^2}{h_{\nu-1}}\bigr).\nonumber
\end{eqnarray}
To justify the inequality above: $\tfrac{\log x}{1-x}$ increases for $x\le 1$, so that
\begin{multline*}
\sum_{k=1}^{\nu-1}\tfrac{\log(k/\nu)}{\nu-k}\le\int_0^1\tfrac{\log x}{1-x}\,dx-\int_0^{1/\nu}\tfrac{\log x}{1-x}\,dx\\
\le -\zeta(2)+\tfrac{\nu}{\nu-1}\int_0^{1/\nu}\log(1/x)\,dx=-\zeta(2)+\tfrac{\log(\nu e)}{\nu-1}.
\end{multline*}

\noindent The big-O term is uniform over all $u\in (0,1)$ and $\nu>1$. It follows then from \eqref{3.0515} that there exist $u(\eps')\in (0,1)$ and $\nu(\eps')>1$ such that \eqref{3.051} holds for $u\in (0,u(\eps'))$ and $\nu\ge \nu(\eps')$.
Furthermore, for $u\in(0,u(\eps'))$ and $\nu\le \nu(\eps')$,
\begin{equation*}
\tfrac{\phi_{\nu}(u)}{\psi_{\nu}(u)}\le A(\eps'):=\tfrac{(1-u(\eps'))^{-\nu(\eps')}}{\exp(u(\eps')\a\log(\nu(\eps'))},
\end{equation*}
since, for $\a\le 1$,  $\tfrac{(1-u)^{-\nu}}{\exp(u\a\log\nu)}$ attains its maximum on $[0,\nu(\eps')]$ at $\nu(\eps')$. 
Combining this inequality with \eqref{3.051}, by induction on $\nu$ we obtain that $\phi_{\nu}(u)\le A(\eps')\psi_{\nu}(u)$ for all $\nu>1$
and $u\le u':= u(\eps')$. 
The rest is easy:
\begin{eqnarray*}
\P\Bigl(D_n\ge \tfrac{6}{\pi^2}(1+\eps)\log n\Bigr)&\le& \tfrac{\E[\exp(u'D_n)]}{\exp\bigl(u' \tfrac{6}{\pi^2}(1+\eps)\log n\bigr)}
\le \tfrac {A(\eps')\psi_{\nu}(u')}{\exp\bigl(u' \tfrac{6}{\pi^2}(1+\eps)\log n\bigr)}
\\
&\le &A(\eps')\exp\Bigl[u' \bigl(\a-\tfrac{6}{\pi^2}(1+\eps)\bigr)\log n\Bigr]=\frac{A(\eps')}{n^{\tfrac{6u'}{\pi^2}(\eps-\eps')}}.
\end{eqnarray*}

{\bf (ii)\/} Predictably, we will use the union bound, which makes it necessary to upper-bound 
$\P\bigl(D_n\ge 2(1+\eps)\log n\bigr)$. To this end, we use a cruder version of the argument in the part {\bf (i)\/}. Set
$\a=1+\eps/2$ and choose $u=\tfrac{1}{\a}$. Denoting $z_{\nu}=u/h_{\nu-1}$ we bound
\begin{multline*}
\tfrac{1}{\psi_{\nu}(u)(h_{\nu-1}-u)}\sum_{k=1}^{\nu-1}\tfrac{\psi_k(u)}{\nu-k}=\tfrac{1}{h_{\nu-1}-u}\sum_{k=1}^{\nu-1}
\tfrac{\exp\bigl(u\a\log(k/\nu)\bigr)}{\nu-k}\\
=\tfrac{h_{\nu-1}}{h_{\nu-1}-u}\cdot\tfrac{1}{h_{\nu-1}}\sum_{k=1}^{\nu-1}\tfrac{k/\nu}{\nu-k}
=\tfrac{h_{\nu-1}}{h_{\nu-1}-u}\cdot\biggl(1-\tfrac{\nu-1}{\nu h_{\nu-1}}\biggr)^{u\a}\\
\le\exp\bigl(-\log(1-z_{\nu})-z_{\nu}\tfrac{\a(\nu-1)}{\nu}\bigr).
\end{multline*}
Since $z_{\nu}\to 0$, the last expression is below $1$ for $\nu\in [\nu(\a),n]$. Therefore, arguing closely to the part {\bf (i)\/}, we see that $\phi_n(u)=O(\psi_n(u))$. Consequently
\[
\P\bigl(D_n\ge 2(1+\eps)\log n\bigr)=O\biggl(\tfrac{\psi_n(u)}{\exp\bigl(2u(1+\eps)\log n\bigr)}\biggr)\\
=O\biggl(n^{-\tfrac{2(1+\eps)}{1+\eps/2}+1}\biggr),
\]
implying, by the union bound, that
\begin{multline*}
\P\bigl(\mathcal D_n\ge 2(1+\eps)\log n\bigr)\le n\P\bigl(D_n\ge 2(1+\eps)\log n\bigr)
=O\Bigl(n^{-\tfrac{2(1+\eps)}{1+\eps/2}+2}\Bigr)=O\Bigl(n^{-\tfrac{\eps}{1+\eps/2}}\Bigr).\\
\end{multline*}
\end{proof}

\subsection{Asymptotic normality of $D_n$.}
\label{sec:Dnorm}   
Here is one part of Theorem \ref{thmG}.
\begin{proposition}\label{0.919} In distribution, and with all of its moments, 
\[
\frac{D_n-\zeta^{-1}(2)\log n}{\sqrt{\tfrac{2\zeta(3)}{\zeta^3(2)}\log n}}\Longrightarrow \mathrm{Normal}(0,1).
\] 
\end{proposition}
\noindent
In particular, this provides a proof of the first-order result
\[  
\var(D_n)\sim\tfrac{2\zeta(3)}{\zeta^3(2)}\log n,
\]
without having to rely on the h-ansatz, 
as stated in Theorem \ref{thmA}.

\begin{proof} By a general theorem due to Curtiss \cite{Cur}, it suffices to show that for $|u|=\Theta(\log^{-1/2}n)$ and properly chosen $\a_1,\,\a_2>0$, the Laplace transform $\phi_n(u)=\E[e^{uD_n}]$ satisfies
\begin{equation}\label{1}
\phi_n(u)=(1+o(1))\exp\bigl[(u\a_1+u^2\a_2)\log n\bigr].
\end{equation}
Recall from \eqref{2} that 
\begin{equation}\label{1.349}
\phi_{\nu}(u)=\tfrac{1}{h_{\nu-1}-u}\sum_{k=1}^{\nu-1}\tfrac{\phi_k(u)}{\nu-k},\quad \nu\ge 2.
\end{equation}
Define a function
\[
\Psi_{\nu}(u)=\exp\bigl[(u\a_1+u^2\a_2)\log \nu\bigr],\quad \nu\in [1,n];
\]
obviously $\Psi_1(u)=1=\phi_1(u)$. 
We will use induction on $\nu$ to prove a stronger result, namely that there exist $\a_1$ and $\a_2$ such that for $|u|=\Theta(\log^{-1/2}n)$, the ratio $\tfrac{\phi_{\nu}(u)}{\Psi_{\nu}(u)}$ converges to $1$, uniformly over $n\ge \nu\to\infty$, sufficiently fast. Pick $\delta\in (0,1/6)$, and set $\nu_n=\lceil \exp( \log^{\delta}n)\rceil$, so in particular $u\log\nu_n\to 0$. Introduce $\Psi^*_{\nu}(u):=1+u\a\log \nu$. 
For $u>0$, 
we have
\begin{multline}\label{onestar}
\tfrac{1}{(h_{\nu-1}-u)\Psi^*_{\nu}(u)}\sum_{k=1}^{\nu-1}\tfrac{\Psi^*_k(u)}{\nu-k}
=\tfrac{1}{h_{\nu-1}}\sum_{j_1,j_2\ge 0}u^{j_1+j_2}\bigl(\tfrac{1}{h_{\nu-1}}\bigr)^{j_1}(-\a\log\nu)^{j_2}\\
\times \biggl((1+u\a\log\nu)h_{\nu-1}+u\a\sum_{k=1}^{\nu-1}\tfrac{\log(k/\nu)}{\nu-k}\biggr)\\
=\Bigl(1+u\bigl(\tfrac{1}{h_{\nu-1}}-\a\log\nu\bigr)+O(\a^2u^2\log^2\nu)\Bigr)
\cdot\bigl(1+u\a\log\nu-\tfrac{u}{h_{\nu-1}}\Theta(\a)\bigr)\\
=1+\tfrac{u}{h_{\nu-1}}(1-\Theta(\a))+O\bigl((\a^2+1)u^2\log^2\nu\bigr)\\
\left\{\begin{aligned}
&>1,&&\text{if }\nu\le\nu_n,\,\a>0\text{ and small},\\
&<1,&&\text{if }\nu\le\nu_n,\,\a>0\text{ and large}.\end{aligned}\right.
\end{multline}
(For the bottom part we used $\delta<\tfrac{1}{6}$.) And the inequalities are interchanged if $u<0$. Combining this with \eqref{1.349}, we conclude that 
$\phi_{\nu}(u)=1+O(|u|\log \nu)=\exp\bigl(O(|u|\log \nu)\bigr)$, uniformly for $\nu\le \nu_n$.
So, for bounded $\a_1,\,\a_2$,
\begin{equation}\label{1.1}
\lim_{n\to\infty}\max_{\nu\le\nu_n}\big|\tfrac{\phi_{\nu}(u)}{\Psi_{\nu}(u)}-1\big|=0.
\end{equation}
Thus, we need to prove existence of $\a_1,\a_2$ such that the property above holds for $\nu\ge \nu_n$, as well.
To this end, let us determine $\a_1$ and $\a_2$ from the condition that $\Psi_{\nu}(u)$, $(\nu\in [\nu_n,n])$, satisfies the recursive inequality
\begin{equation}\label{2.32}
\Psi_{\nu}(u)\ge (\le )\tfrac{1}{h_{\nu-1}-u}\biggl(\,\sum_{k=1}^{\nu-1}\tfrac{\Psi_k(u)}{\nu-k}\biggr),\quad \nu\in [\nu_n,n].
\end{equation}
First of all, we have
\[
\Psi_k(u)=\Psi_{\nu}(u)\exp\bigl[\bigl(u\a_1 +u^2\a_2\bigr)\log(k/\nu)\bigr],\quad k\le\nu.
\]
Therefore
\begin{multline}\label{1.23}
\tfrac{1}{\Psi_{\nu}(u)(h_{\nu-1}-u)}\sum_{k=1}^{\nu-1}\tfrac{\Psi_k(u)}{\nu-k}=\tfrac{1}{h_{\nu-1}-u}\sum_{k=1}^{\nu-1}
\tfrac{\exp\bigl[\bigl(u\a_1 +u^2\a_2\bigr)\log(k/\nu)\bigr]}{\nu-k}\\
=\bigl(1-\tfrac{u}{h_{\nu-1}}\bigr)^{-1}\cdot\Bigl(1+\tfrac{1}{h_{\nu-1}}
\sum_{k=1}^{\nu-1}\tfrac{\exp\bigl[\bigl(u\a_1 +u^2\a_2\bigr)\log(k/\nu)\bigr]-1}{\nu-k}\Bigr)\\
=\bigl(1-\tfrac{u}{h_{\nu-1}}\bigr)^{-1}\cdot\biggl(1+\tfrac{1}{h_{\nu-1}}\int_0^1\tfrac{\exp\bigl[\bigl(u\a_1+u^2\a_2\bigr)
\log x\bigr]-1}{1-x}\,dx+O\bigl(\tfrac{|u|\log\nu_n}{\nu_n}\bigr)\biggr).
\end{multline}
In the final line, the bottom integral does not depend on $\nu$. 
Let us first justify the remainder term. Define $f(k/\nu)$ as the $k$-th term in the previous sum, ($k<\nu$), and, for continuity, set $f(\nu/\nu)=-\nu^{-1}(u\a_1+u^2\a_2)$.  It can be checked that 
$f^{\prime\prime}_k(k/\nu)$ does not change its sign on $[1,\nu]$. So, replacing the sum with the integral for $k$ varying continuously from $1$ to $\nu$, we introduce the error on the order of the sum of absolute values of 
\[
f(k/\nu)\Big|^{\nu}_1\quad\text{and}\quad f'_k(k/\nu)\Big|^{\nu}_1.
\]
The dominant contribution to each of these terms comes from $k=1$. Since for $\sigma\in (0,1)$  the function $\tfrac{z^{\sigma}-1}{z}$ decreases for $z\ge \sigma^{-1}\log\tfrac{1}{1-\sigma}$, we bound
\[
|f(1/\nu)|\le \tfrac{\exp(|u\a_1+u^2\a_2|\log \nu_n)-1}{\nu_n}=O\bigl(\nu_n^{-1}|u|\log\nu_n\bigr).
\]
And the bound for $|f'_k(1/\nu)|$ is even better. So the sum in question is of order $O\bigl(\tfrac{|u|\log\nu_n}{\nu_n}\bigr)$
uniformly for $\nu\ge \nu_n$. Extending the resulting integral to the full $[0,\nu]$, we introduce the second error on the order of
\begin{equation}\label{1.235}
\int_0^1\tfrac{\exp\bigl[\bigl(u\a_1+u^2\a_2\bigr)\log (k/\nu)\bigr]-1}{\nu-k}\,dk=O\bigl(\tfrac{|u|\log\nu_n}{\nu_n}\bigr).
\end{equation}
The sum of the two error terms is  $O\bigl(\nu^{-1}|u|\log\nu)$, and dividing it by $h_{\nu-1}$ we get $O\bigl(\tfrac{|u|}{\nu_n}\bigr)$.

Let us sharply estimate  the bottom integral in \eqref{1.23}.  By \eqref{1.235}, the contribution to this integral coming from $x\in (0,1/\nu_n]$ is $O(\nu^{-1}_n|u|\log\nu_n)$. And for $x\in [1/\nu_n,1]$, we have $|u|\log(1/x)\le |u|\log\nu_n\to 0$,  i.e. we can use the Taylor expansion
\begin{multline*}
\tfrac{\exp\bigl[\bigl(u\a_1+u^2\a_2\bigr)\log x\bigr]-1}{1-x}
=\tfrac{\bigl(u\a_1+u^2\a_2\bigr)\log x}{1-x}+\tfrac{\bigl(u\a_1+u^2\a_2\bigr)^2\log^2 x}{2(1-x)}
+O\bigl(\tfrac{|u|^3\log^3(1/x)}{1-x}\bigr).\\
\end{multline*}
This means that, at the price of the error term of the order $\nu^{-1}_n |u|\log\nu_n+|u|^3\int_0^1\tfrac{\log^3(1/x)}{1-x}\,dx$,
we can use the expansion above for all $x\in (0,1]$. 

So, using \eqref{2.495}, we obtain

\begin{multline*}
\tfrac{1}{h_{\nu-1}}\int_0^1\tfrac{\exp\bigl[\bigl(u\a_1+u^2\a_2\bigr)\log x\bigr]-1}{1-x}\,dx\\
=-\tfrac{\a_1\zeta(2)u}{h_{\nu-1}}+\tfrac{u^2}{h_{\nu-1}}\bigl[\a_1^2\zeta(3)-\a_2\zeta(2)\bigr] +O\bigl(\tfrac{|u|^3}{h_{\nu-1}}+\nu^{-1}_n |u|\bigr).
\end{multline*}
Consequently, for $\nu\ge \nu_n\bigl(=\lceil \exp(\log^{\delta} n)\rceil)$,
\begin{multline}\label{2.34}
\tfrac{1}{\Psi_{\nu}(u)(h_{\nu-1}-u)}\sum_{k=1}^{\nu-1}\tfrac{\Psi_k(u)}{\nu-k}
=1+\tfrac{u}{h_{\nu-1}}\bigl(1-\a_1\zeta(2)\bigr)\\
+\tfrac{u^2}{h_{\nu-1}}\bigl[\a_1^2\zeta(3)-\a_2\zeta(2)\bigr]+O\bigl(\tfrac{|u|^3}{h_{\nu-1}}+\tfrac{|u|}{\nu_n}\bigr)\\
=1+\tfrac{u}{h_{\nu-1}}\bigl(1-\a_1\zeta(2)\bigr)+O\bigl(\tfrac{|u|^3}{h_{\nu-1}}\bigr),
\end{multline}
if we select $\a_2=\tfrac{\a_1^2\zeta(3)}{\zeta(2)}$, which we certainly do. Suppose $u>0$; set $\a_1=\zeta^{-1}(2)+u^{b}$, $b\in (1,2)$. Then, uniformly for $\nu\in [\nu_n,n]$, we have
\begin{equation*}
1+\tfrac{u}{h_{\nu-1}}\bigl(1-\a_1\zeta(2)\bigr)+O\bigl(\tfrac{|u|^3}{h_{\nu-1}}\bigr)=1-\tfrac{\zeta^{-1}(2)u^{b+1}}{h_{\nu-1}}
\bigl(1+O(u^{2-b}))<1.
\end{equation*}
So, \eqref{2.34} becomes
\begin{equation*}
\tfrac{1}{h_{\nu-1}-u}\sum_{k=1}^{\nu-1}\tfrac{\Psi_k(u)}{\nu-k}\le \Psi_{\nu}(u).
\end{equation*}
This equation and the equation \eqref{1.349} together imply, by induction on $\nu\in [\nu_n,n]$, that 
$\limsup_{n\to\infty} \max_{\nu\in [\nu_n,n]}\tfrac{\phi_{\nu}(u)}{\Psi_{\nu}(u)}\le 1$. Now,
\begin{multline*}
\Psi_{\nu}(u)=\exp\bigl[(u\a_1+u^2\a_2)\log \nu\bigr]\\
=\exp\Bigl[\bigl(u\zeta^{-1}(2)+u^2\tfrac{\zeta(3)}{\zeta^3(2)}\bigr)\log\nu +O(u^{b+1}\log \nu)\Bigr]\\
\sim\exp\Bigl[\bigl(u\zeta^{-1}(2)+u^2\tfrac{\zeta(3)}{\zeta^3(2)}\bigr)\log \nu\Bigr],
\end{multline*}
since $u^{b+1}\log n=O\Bigl(\log^{-\tfrac{b-1}{2}}n\Bigr)$ and $b>1$. Therefore 
\[
\limsup_{n\to\infty}\max_{\nu\in [\nu_n,n]}\tfrac{\phi_{\nu}(u)}{\Psi_{\nu}(u)}\le 1.
\]
 Analogously, setting $\a_1=\zeta^{-1}(2)-u^b$, we have 
 \[
 \liminf_{n\to\infty}\min_{\nu\in [\nu_n,n]}\tfrac{\phi_{\nu}(u)}{\Psi_{\nu}(u)}\ge 1.
 \]
So, for $u=\Theta(\log^{-1/2}n)>0$ we have 
\[
\lim_{n\to\infty}\tfrac{\phi_n(u)}{\exp\Bigl[\bigl(u\zeta^{-1}(2)+u^2\tfrac{\zeta(3)}{\zeta^3(2)}\bigr)\log n\Bigr]}=1.
\]
The case $u<0$ is completely similar, so that the last equation holds for $u=-\Theta(\log^{-1/2}n)<0$ as well.
\end{proof}

\subsection{The moments of edge-heights of the leaves} 
\label{sec:Lmoments}
Recall that $L_n$ denotes the edge-height of a uniform random leaf. 
In this section we  prove Theorem \ref{thmB} via the two Propositions below.
\begin{proposition}\label{2new} 
\begin{equation}\label{2.1}
\E[L_n]=\tfrac{1}{2\zeta(2)}\log^2 n+\tfrac{\ga\,\zeta(2)+\zeta(3)}{\zeta^2(2)}\log n+O(1).
\end{equation}
\end{proposition}
\begin{proof} 
The straightforward recurrence  for $\E[L_{\nu}]$ is
\begin{equation}\label{2.3}
\E[L_{\nu}]=1+\tfrac{1}{h_{\nu-1}}\sum_{k=1}^{\nu-1}\tfrac{\E[L_k]}{\nu-k}.                            
\end{equation}
Write $\E[L_{\nu}]=A \log^2 \nu+B\log \nu +u_{\nu}$, so that $u_1=0$. We need to show that $u_{\nu}=O(1)$, if we select
$A$ and $B$ appropriately.  (Sure enough, these will be the constants in the claim.) Using \eqref{2.3}, we have
\begin{multline}\label{2.37}
u_{\nu}=1+\tfrac{1}{h_{\nu-1}}\sum_{k\in [\nu-1]}\tfrac{u_k}{\nu-k}+A\biggl(\tfrac{1}{h_{\nu-1}}\sum_{k\in [\nu-1]}\tfrac{\log^2 k}{\nu-k}-\log^2\nu\biggr)\\
+B\biggl(\tfrac{1}{h_{\nu-1}}\sum_{k\in[\nu-1]}\tfrac{\log k}{\nu-k}-\log \nu\biggr).
\end{multline}
Here, by \eqref{2.034},
\begin{multline*}
\tfrac{1}{h_{\nu-1}}\sum_{k\in[\nu-1]}\tfrac{\log k}{\nu-k}-\log \nu=\tfrac{1}{h_{\nu-1}}\sum_{k\in [\nu-1]}\tfrac{\log(k/\nu)}{\nu-k}
=-\tfrac{\zeta(2)}{h_{\nu-1}}+\tfrac{\log(2\pi e)}{\nu h_{\nu-1}} +O(\nu^{-2}),\\
\end{multline*}
and, combining the equation above with \eqref{2.0315}, we also have
\begin{multline*}
\tfrac{1}{h_{\nu-1}}\sum_{k\in [\nu-1]}\tfrac{\log^2 k}{\nu-k}-\log^2\nu=\tfrac{1}{h_{\nu-1}}\sum_{k\in [\nu-1]}
\tfrac{\log(k/\nu)\cdot(\log(k/\nu)+2\log\nu)}{\nu-k}\\
=\tfrac{2\zeta(3)}{h_{\nu-1}}+O(\nu^{-1}\log\nu)+2\bigl(-\tfrac{\zeta(2)\log\nu}{h_{\nu-1}}+\tfrac{\log(2\pi e)\log\nu}{\nu h_{\nu-1}}+
O(\nu^{-2}\log\nu)\biggr).\\
\end{multline*}
Plugging the estimates above into \eqref{2.37} and using $\log\nu=h_{\nu-1}-\ga+O(\nu^{-1})$, we get 
\begin{multline*}
u_{\nu}=\tfrac{1}{h_{\nu-1}}\sum_{k\in [\nu-1]}\tfrac{u_k}{\nu-k}
+(1-2A\zeta(2))+\tfrac{1}{h_{\nu-1}}\bigl[2A(\ga\zeta(2)+\zeta(3))-B\zeta(2)\bigr]+O(\nu^{-1}\log\nu).\\
\end{multline*}
So, selecting $A$ and $B$ such that the $(A,B)$-dependent coefficients are both zeros, i. e. $A=\tfrac{1}{2\zeta(2)}$, $B=\tfrac{\ga\zeta(2)+\zeta(3)}{\zeta(2)}$, we arrive at
\[
u_{\nu}=\tfrac{1}{h_{\nu-1}}\biggl(\sum_{k\in [\nu-1]}\tfrac{u_k}{\nu-k}+O(\nu^{-1}\log^2\nu)\biggr).
\]
From the proof of Proposition \ref{thm0.91} (starting with \eqref{3.039}), it follows that $u_{\nu}=O(1)$.
\end{proof}

\begin{proposition}\label{new3prop}
$
\var(L_n) =\tfrac{2\zeta(3)}{3\zeta^3(2)}\log^3 n+O(1).
$
\end{proposition}
\begin{proof} {\bf (i)\/} The key is 
\begin{lemma}\label{lem0.91} Setting $\bv_n:=\var(L_n)$, we have
\begin{equation}\label{new3}
\bv_n=-1+\tfrac{1}{h_{n-1}}\sum_{k=1}^{n-1}\tfrac{\bv_k+(\E[L_n]-\E[L_k])^2}{n-k}.
\end{equation}
\end{lemma}
\noindent {\bf Note.\/} In particular, $\bv_2=0$ as it should be, since $L_2\equiv 1$, unlike $D_2$ which is distributed exponentially with rate $1$.
\begin{proof} 
Differentiating twice both sides of \eqref{2.2} at $u=0$, we get 
\begin{multline*}
\E[L_{\nu}^2]=1+\tfrac{1}{h_{\nu-1}}\sum_{k\in [\nu-1]}\tfrac{\E[L_k^2]}{\nu-k}+\tfrac{2}{h_{\nu-1}}\sum_{k\in [\nu-1]}\tfrac{\E[L_k]}{\nu-k}\\
=2\E[L_{\nu}]-1+\tfrac{1}{h_{n-1}}\sum_{k=1}^{n-1}\tfrac{\E[L_k^2]}{n-k}\\
=2\E[L_{\nu}]-1+\tfrac{1}{h_{\nu-1}}\sum_{k=1}^{\nu-1}\tfrac{\E^2[L_k]}{\nu-k}+\tfrac{1}{h_{\nu-1}}\sum_{k=1}^{\nu-1}\tfrac{\bv_k}{\nu-k}.
\end{multline*}
Since $\bv_{\nu}=\E[L_{\nu}^2]-\E^2[L_{\nu}]$, the equation above becomes
\[
\bv_{\nu}=2E[L_{\nu}]-1+\tfrac{1}{h_{n-1}}\sum_{k=1}^{\nu-1}\tfrac{\bv_k+\E^2[L_k]}{\nu-k}-\E^2[L_{\nu}],
\]
and it is easy to check that this equation is equivalent to the claim.
\end{proof}

{\bf (ii)\/} Using Proposition \ref{2new}, we compute, for
 $A=\tfrac{1}{2\zeta(2)}$, $B=\tfrac{\ga\zeta(2)+\zeta(3)}{\zeta(2)}$,
\begin{multline*}
\bigl(\E[L_{\nu}]-\E[L_k]\bigr)^2=\Bigl(A\bigl(\log^2\nu-\log^2k\bigr)+B\bigl(\log\nu-\log k\bigr)+O(1)\Bigr)^2\\
=\bigl[2A(\log(k/\nu))\log\nu\bigr]^2+O\bigl[\mathcal{P}(\log(\nu/k))\log \nu\bigr],
\end{multline*}
where $\mathcal{P}(\eta)$ is a fourth-degree polynomial. Therefore, invoking \eqref{2.0315}, we have
\begin{multline*}
\tfrac{1}{h_{\nu-1}}\sum_{k=1}^{\nu-1}\tfrac{(\E[L_{\nu}]-\E[L_k])^2}{\nu-k}=\tfrac{4A^2\log^2\nu}{h_{\nu-1}}\sum_{k=1}^{\nu-1}\tfrac{\log^2(k/\nu)}{\nu-k}+O(1)\\
= \tfrac{8A^2\zeta(3)\log^2\nu}{h_{\nu-1}}+O(1)=8A^2\zeta(3)\log\nu+O(1).
\end{multline*}
So, since $A=\tfrac{1}{2\zeta(2)}$, the equation \eqref{new3} becomes
\begin{equation}\label{new3.5}
\bv_{\nu}=\tfrac{2\zeta(3)}{\zeta(2)^2}\log\nu+O(1)+\tfrac{1}{h_{\nu-1}}\sum_{k=1}^{\nu-1}\tfrac{\bv_k}{\nu-k}.
\end{equation}
Let us use this recurrence to show that, for appropriately chosen $\bA$, 
\[
\bv_\nu={\mathcal V}_{\nu}+O(1),\quad {\mathcal V}_{\nu}:=\bA \log^3\nu.
\]
Here $O(1)$ is uniform over all $\nu\ge 2$. We compute
\begin{multline*}
\tfrac{1}{h_{\nu-1}}\sum_{k=1}^{\nu-1}\tfrac{\log^3 k}{\nu-k}=\tfrac{1}{h_{\nu-1}}
\sum_{k=1}^{\nu-1}\tfrac{\bigl(\log \nu+\log(k/\nu)\bigr)^3}{\nu-k}\\
=\tfrac{1}{h_{\nu-1}}\biggl(\log^3\nu\, h_{\nu-1}+3\log^2\nu\sum_{k=1}^{\nu-1}\tfrac{\log(k/\nu)}{\nu-k}\\
\qquad\qquad+3\log\nu\sum_{k=1}^{\nu-1}\tfrac{\log^2(k/\nu)}{\nu-k}+\sum_{k=1}^{\nu-1}\tfrac{\log^3(k/\nu)}{\nu-k}\biggr)\\
=\log^3\nu+\tfrac{3\log^2\nu}{h_{\nu-1}}\sum_{k=1}^{\nu-1}\tfrac{\log(k/\nu)}{\nu-k}+O(1)\\
=\log^3\nu-3\zeta(2)\log\nu+O(1).
\end{multline*}
It follows that 
\begin{align*}
\tfrac{2\zeta(3)}{\zeta(2)^2}\log\nu+&\tfrac{1}{h_{\nu-1}}\sum_{k=1}^{\nu-1}\tfrac{\mathcal V_k}{\nu-k}\\
&=\mathcal V_{\nu}+\Bigl(\tfrac{2\zeta(3)}{\zeta(2)^2}-3\bA\zeta(2)\Bigr)\log\nu+O(1) =\mathcal V_{\nu}+O(1),
\end{align*}
if we select $\bA=\tfrac{2\zeta(3)}{3\zeta^3(2)}$. Combining this equation with \eqref{new3.5}, and using induction we
obtain that $|\bv_{\nu}-\mathcal V_{\nu}|\le C$ for some absolute constant $C$.
\end{proof}

\subsection{Bounding the edge-height of the random tree.}
\label{sec:Dtail}
As with the time-height in section \ref{sec:Rtail}, we use a tail bound on the edge-height of a random leaf to obtain a tail bound for the edge-height of the tree itself.

\begin{proposition}\label{0.1} Let $L_n$ denote the edge-height of the uniformly random leaf, and let $\mathcal L_n$ denote the largest edge-height of a leaf. 

\noindent
{\bf (1)\/} For $\eps>0$,
\[
\P\bigl(L_n\ge\tfrac{3}{\pi^2}(1+\eps)\log^2n\bigr)=O\bigl(n^{-\Theta(\eps)}\bigr).
\]
 {\bf (2)\/} 
 Let $\be=\min_{\a>1/\log 2}\bigl[\a+\tfrac{4\a^2\zeta(3)}{\a\log 2-1}\bigr]\approx 42.9$. For $\eps\in (0,1)$, 
\[
\P\Bigl(\mathcal L_n\ge (1+\eps)\be\log^2 n)\Bigr)\le \exp\bigl(-\Theta(\eps\log n)\bigr).
\]
\end{proposition}

\begin{proof} {\bf (1)\/} First of all, $f_{\nu}(u):=\E[e^{u L_{\nu}}]\le e^{u(\nu-1)}$ for $u\ge 0$. By \eqref{2.2}, we have:  
\[
f_{\nu}(u):=\E[e^{u L_{\nu}}]=\tfrac{e^{u}}{h_{\nu-1}}\sum_{k=1}^{\nu-1}\tfrac{f_k(u)}{\nu-k}, \quad \nu\in [2,n].
\]

Consider $u=\tfrac{v}{\log n}$, and introduce $g_k(u)=\exp(u\a\log^2k)$, $\a>0$ yet to be determined. 
Let us prove that there exists $v=v(\a)$ sufficiently small, and $\nu(\a)$ sufficiently large such that
\begin{equation}\label{4.051}
g_{\nu}(u)\ge\tfrac{e^u}{h_{\nu-1}}\sum_{k=1}^{\nu-1}\tfrac{g_k(u)}{\nu-k},\quad\forall \nu\in [\nu(a),n],
\end{equation}
We compute
\begin{multline}\label{4.0511}
\tfrac{1}{g_{\nu}(u) h_{\nu-1}}\sum_{k=1}^{\nu-1}\tfrac{g_k(u)}{\nu-k}
=\tfrac{1}{h_{\nu-1}}\sum_{k=1}^{\nu-1}\tfrac{\exp[u\a(\log^2k-\log^2\nu)]}{\nu-k}\\
=1+\tfrac{1}{h_{\nu-1}}\sum_{k=1}^{\nu-1}\tfrac{\exp[u\a(\log^2k-\log^2\nu)]-1}{\nu-k}\\
\le 1+\tfrac{u\a}{h_{\nu-1}}\sum_{k=1}^{\nu-1}\tfrac{\log^2k-\log^2\nu}{\nu-k}
+O\biggl(u^2\log \nu\sum_{k=1}^{\nu-1}\tfrac{\log^2(k/\nu)}{\nu-k}\!\biggr).\\
\end{multline}
The big-Oh term is $O(u^2\log \nu)$, and
\begin{multline*}
\sum_{k=1}^{\nu-1}\tfrac{\log^2k-\log^2\nu}{\nu-k}=\sum_{k=1}^{\nu-1}\tfrac{\log^2(k/\nu)}{\nu-k}+
2\log\nu\sum_{k=1}^{\nu-1}\tfrac{\log(k/\nu)}{\nu-k}=-\tfrac{\pi^2}{3}\log \nu +O(1).\\
\end{multline*}
Therefore
\begin{equation*}
\tfrac{1}{g_{\nu}(u) h_{\nu-1}}\sum_{k=1}^{\nu-1}\tfrac{g_k(u)}{\nu-k}\le 1-\tfrac{\a u\pi^2}{3}+O\bigl(u\log^{-1}\nu+
u^2\log \nu\bigr),
\end{equation*}
implying that 
\begin{equation}\label{4.0512}
\tfrac{e^u}{g_{\nu}(u) h_{\nu-1}}\sum_{k=1}^{\nu-1}\tfrac{g_k(u)}{\nu-k}
\le 1-u\bigl(\tfrac{\a\pi^2}{3}-1\bigr)+O\bigl(u\log^{-1}\nu+u^2\log \nu\bigr).
\end{equation}
Recalling that $u=\tfrac{v}{\log n}$, we obtain that for $\a>\tfrac{3}{\pi^2}$ there exists a sufficiently small $v(\a)>0$, and
a sufficiently large $\nu(\a)$, such that for $v\le v(\a)$ and $n\ge \nu\ge \nu(\a)$, we have
\begin{equation}\label{4.0513}
\tfrac{e^u}{g_{\nu}(u) h_{\nu-1}}\sum_{k=1}^{\nu-1}\tfrac{g_k(u)}{\nu-k}\le 1.
\end{equation}
Furthermore, for $u\le \tfrac{v(\a)}{\log n}$ and $\nu\le \nu(\a)$, 
\begin{equation}\label{4.0514}
\tfrac{f_{\nu}(u)}{g_{\nu}(u)}\le \tfrac{\exp(u\nu)}{\exp(u\a\log^2\nu)}\le \exp\bigl[\tfrac{v(\a)\nu(\a)}{\log n}\bigr].
\end{equation}
By induction on $\nu\in[\nu(a),n]$, it follows that for those $\nu$'s
\[
f_{\nu}(u)\le \exp\bigl[\tfrac{v(\a)\nu(\a)}{\log n}\bigr] g_{\nu}(u)\Longrightarrow f_n(u)\le\exp\bigl[\tfrac{v(\a)\nu(\a)}{\log n}\bigr] g_n(u),
\]
provided that $\a>\tfrac{3}{\pi^2}$, and $u\le \tfrac{v(\a)}{\log n}$. So, given $\eps>0$, we set $\a=\tfrac{3}{\pi^2}(1+\eps/2)$, $\a'=\tfrac{3}{\pi^2}(1+\eps)$, and bound
\begin{multline*}
\P\bigl(L_n\ge\tfrac{3}{\pi^2}(1+\eps)\log^2n\bigr)=O\bigl(\tfrac{\exp(u\a\log^2n)}{\exp(u\a'\log^2n)}\bigr)\Big|_{u=\tfrac{v(\a)}{\log n}}
=O\bigl(n^{-\Theta(\eps)}\bigr).\\
\end{multline*}
This is the assertion of {\bf (1)}.

{\bf (2)\/} We need a more explicit, but cruder, version of \eqref{4.0512}, again for $u=O(\log^{-1}n)$.  Instead of \eqref{4.0511}, we bound
\[
\tfrac{1}{g_{\nu}(u) h_{\nu-1}}\sum_{k=1}^{\nu-1}\tfrac{g_k(u)}{\nu-k}\le
 1+\tfrac{u\a}{h_{\nu-1}}\sum_{k=1}^{\nu-1}\tfrac{\log^2k-\log^2\nu}{\nu-k}+\tfrac{2\a^2u^2\log^2\nu}{h_{\nu-1}}\sum_{k=1}^{\nu-1}\tfrac{\log^2(k/\nu)}{\nu-k}.
 \]
Here 
\begin{multline*}
\tfrac{1}{h_{\nu-1}}\sum_{k=1}^{\nu-1}\tfrac{\log^2k-\log^2\nu}{\nu-k}=\tfrac{1}{h_{\nu-1}}\sum_{k=1}^{\nu-1}
\tfrac{\log^2(k/\nu)}{\nu-k}+\tfrac{2\log\nu}{h_{\nu-1}}\sum_{k=1}^{\nu-1}\tfrac{\log(k/\nu)}{\nu-k}\\
\le \tfrac{2\zeta(3)}{h_{\nu-1}}+2\log \nu\cdot \log\bigl(1-\tfrac{\nu-1}{\nu h_{\nu-1}}\bigr)\le 
\tfrac{\zeta(3)}{\log\nu}-2\tfrac{(\nu-1)\log\nu}{h_{\nu-1}\nu}
\le \tfrac{2\zeta(3)}{\log\nu}-\log 2,\\
\end{multline*}
since $\tfrac{\log^2 x}{1-x}$ is increasing, $\log x$ is concave, and $\log(1+z)\le z$, ($z>-1$). So, we replace \eqref{4.0512} with
\begin{equation}\label{4.0515}
\tfrac{e^u}{g_{\nu}(u) h_{\nu-1}}\sum_{k=1}^{\nu-1}\tfrac{g_k(u)}{\nu-k}
\le 1-u\Bigl(\a\bigl(\log 2-\tfrac{2\zeta(3)}{\log\nu}\bigr)-1\Bigr)+u^2(1+4\a^2\zeta(3)\log\nu).
\end{equation}
Given $\a>0$, the coefficient by $u$ can be made arbitrarily close to $\a\log 2-1$ for $\nu$ sufficiently large, thus positive if
$\a>\tfrac{1}{\log 2}$. Assuming the latter, for those large $\nu$'s, still below $n$,  the RHS expression in \eqref{4.0515}
is below $1$  if $0< u\le \tfrac{\a\log 2-1}{(4+\delta)\a^2\zeta(3)\log n}$, $(\delta>0)$, and $n\ge n(\delta)$. It follows that, as 
$n\to\infty$, we have $f_n(u)=O(g_n(u))$. Consequently, given $\a'>\a>\tfrac{1}{\log 2}$, 
\begin{equation}\label{4.0516}
\begin{aligned}
\P(\mathcal L_n\ge \a'\log^2n)&\le n \P(L_n\ge \a'\log^2 n)=O\bigl(n\tfrac{e^{\a u \log^2n}}{e^{\a'u\log^2n}}\bigr)\\
&=O\Bigl(\exp\bigl[\log n-u(\a'-\a)\log^2 n)\bigr]\Bigr)\to 0,
\end{aligned}
\end{equation}
if 
\[
u=\tfrac{\a\log 2-1}{(4+\delta)\a^2\zeta(3)\log n},\quad \a'>\a+\tfrac{(4+\delta)\a^2\zeta(3)}{\a\log 2-1}.
\]
Set 
\[
\be=\min_{\a>1/\log 2}\bigl[\a+\tfrac{4\a^2\zeta(3)}{\a\log 2-1}\bigr],
\]
and let $\hat\a$ stand for the point where the minimum is attained. Given $\eps>0$, there exists $\delta=\delta(\eps)=\Theta(\eps)$ such that
\[
\a':=(1+\eps)\be > \hat \a+\tfrac{(4+\delta)\hat\a^2\zeta(3)}{\hat\a\log 2-1},
\]
implying, by \eqref{4.0516} with $\a=\hat\a$, that $\P(\mathcal L_n\ge \a'\log^2n)=O\bigl(\exp(-\Theta(\eps)\log n)\bigr)$.
\end{proof}

\subsection{Asymptotic normality of $L_n$.}
\label{sec:Lnorm}
Here is the second part of Theorem \ref{thmG}.
\begin{proposition}\label{0.919new} In distribution, and with all of its moments, 
\[
\frac{L_n-(2\zeta(2))^{-1}\log^2 n}{\sqrt{\tfrac{2\zeta(3)}{3\zeta^3(2)}\log^3 n}}\Longrightarrow \mathrm{Normal}(0,1).
\]. 
\end{proposition}
\begin{proof} 
Analogously to
the proof of Proposition \ref{0.919}, it would seem natural to show that for $|u|=\Theta(\log^{-3/2}n)$ and properly chosen $\a_1>0,\,\a_2>0$, the Laplace transform $f_{\nu}(u)=\E[e^{uL_{\nu}}]$ satisfies
\begin{equation}\label{1new}
f_{\nu}(u)=(1+o(1))g_{\nu}(u),\quad g_{\nu}(u):= \exp\bigl(u\a_1\log^2 \nu+u^2\a_2\log^3 \nu\bigr),
\end{equation}
uniformly for $\nu\le n$. But we could only prove \eqref{1new} for $0<u\le v \log^{-3/2}n$ and a fixed $v>0$. Fortunately, that's all we need, thanks to a relatively recent extension of the Curtis
Lemma: it suffices to prove convergence of the sequence of Laplace transforms for the parameter $v$ confined to a fixed interval $(0,\sigma]$, see 
\cite{MRS} or  \cite{Yak}.

Recall 
\begin{equation}\label{1.349new}
f_{\nu}(u)=\tfrac{e^u}{h_{\nu-1}}\sum_{k=1}^{\nu-1}\tfrac{f_k(u)}{\nu-k},\quad \nu\ge 2.
\end{equation}
Pick $\delta\in (0,3/4)$ and set $\nu_n=\lceil\exp(\log^{\delta}n)\rceil$.
For a constant $\a$, introduce $g^*_{\nu}(u):=\exp(u\a \log^2\nu)$. For $u>0$, we have
\begin{multline}\label{twostar}
\tfrac{e^u}{h_{\nu-1}g^*_{\nu}(u)}\sum_{k=1}^{\nu-1}\tfrac{g^*_k(u)}{\nu-k}
=\tfrac{e^u}{h_{\nu-1}}\sum_{k=1}^{\nu-1}\tfrac{\exp\bigl(u\a\log (k\nu)\log(k/\nu)\bigr)}{\nu-k}\\
=\tfrac{e^u}{h_{\nu-1}}\sum_{k=1}^{\nu-1} \tfrac{1}{\nu-k} \Bigl[1+u\a\log(k\nu)\log(k/\nu)+O\bigl(u^2\a^2\log^2\nu\log^2(k/\nu)\bigr)\Bigr]\\
=e^u\Bigl[1+\tfrac{u\a\Theta(\log\nu)}{h_{\nu-1}}\sum_{k=1}^{\nu-1}\tfrac{\log(k/\nu)}{\nu-k}+O\bigl(u^2\a^2\log^2\nu\sum_{k=1}^{\nu-1}\tfrac{\log^2(k/\nu)}{\nu-k}\bigr)\Bigr]\\
=e^u\bigl[1-u\Theta(\a)+O(u^2\a^2\log^2\nu)\bigr]=1+u(1-\Theta(\a))+O(u^2(1+\a^2\log\nu))\\
\left\{\begin{aligned}
&>1,&&\text{if }\nu\le\nu_n,\,\a>0\text{ and small},\\
&<1,&&\text{if }\nu\le\nu_n,\,\a>0\text{ and large}.\end{aligned}\right.
\end{multline}
(For the second line we used $e^z=1+z+O(z^2/2)$, uniformly for $z<0$. For the bottom line we used $\delta<\tfrac{3}{4}$.) 
Combining this with \eqref{1.349new}, we conclude that 
$f_{\nu}(u)=\exp\bigl(O(u\log^2 \nu)\bigr)$, uniformly for $\nu\le \nu_n$.
So, for bounded $\a_1,\,\a_2$,
\begin{equation}\label{1.1new}
\lim_{n\to\infty}\max_{\nu\le\nu_n}\big|\tfrac{f_{\nu}(u)}{g_{\nu}(u)}-1\bigr|=0.
\end{equation}
Thus, we need to prove existence of $\a_1,\a_2$ such that the analogous relation holds uniformly for all $\nu\in [\nu_n, n]$.
Predictably, we select $\a_1$ and $\a_2$, requiring that $g_{\nu}(u)$ is the asymptotically best fit for the recurrence 
\eqref{1.349new} for all $\nu\in [\nu_n,n]$. To begin,
\begin{equation}\label{1.11new}
\begin{aligned}
g_k(u)&=g_{\nu}(u)\exp\bigl[u\a_1 G_1(k/\nu,\nu)+u^2\a_2G_2(k/\nu,\nu)\bigr],\\
G_1(k/\nu,\nu)&:=\log\bigl(\tfrac{k}{\nu}\bigr)\log(k\nu)\le 0,\\
G_2(k/\nu,\nu)&:=\log\bigl(\tfrac{k}{\nu}\bigr)\bigl[3\log k\log\nu+\log^2\bigl(\tfrac{k}{\nu}\bigr)\bigr]\le 0.
\end{aligned}
\end{equation}
And, since $G_j(k/\nu,\nu)$ are {\it non-positive\/}, the Taylor  
expansion of the exponential function holds for $u>0$, even though $|uG_1(1/n, n)|=O(\sqrt{\log n})$. (Notice that 
$u^2 |G_2(1/\nu,\nu)|=O(1)$.)
So, proceeding analogously to \eqref{1.23},
\begin{multline}\label{1.23new}
\tfrac{e^u}{g_{\nu}(u)h_{\nu-1}}\sum_{k=1}^{\nu-1}\tfrac{g_k(u)}{\nu-k}=\tfrac{e^u}{h_{\nu-1}}\sum_{k=1}^{\nu-1}
\tfrac{\exp\bigl[u\a_1 G_1(k/\nu,\nu)+u^2\a_2G_2(k/\nu,\nu)\bigr]}{\nu-k}\\
=e^u\cdot\Bigl(1+\tfrac{1}{h_{\nu-1}}
\sum_{k=1}^{\nu-1}\tfrac{\exp\bigl[u\a_1 G_1(k/\nu,\nu)+u^2\a_2G_2(k/\nu,\nu)\bigr]-1}{\nu-k}\Bigr)\\
=e^u\cdot\biggl(1+\tfrac{1}{h_{\nu-1}}\int_0^1\tfrac{\exp\bigl[u\a_1G_1(x,\nu)+u^2\a_2 G_2(x,\nu)\bigr]-1}{1-x}\,dx+O\bigl(\tfrac{\exp(-u\log^2\nu\bigr)}{\nu\log\nu}\biggr),
\end{multline}
where {\it uniformly\/} for $x\in (0,1]$:	
\begin{multline*}
\tfrac{\exp\bigl[u\a_1G_1(x,\nu)+u^2\a_2 G_2(x,\nu)\bigr]-1}{1-x}
=\tfrac{u\a_1G_1(x,\nu)+u^2\a_2 G_2(x,\nu)}{1-x}\\
+\tfrac{\bigl(u\a_1G_1(x,\nu)+u^2\a_2 G_2(x,\nu)\bigr)^2}{2(1-x)}+O\bigl(\tfrac{u^3\log^3(1/x)\log^3\nu}{1-x}\bigr).
\end{multline*}
Using \eqref{1.11new}, and \eqref{2.495}, we have then
\begin{multline*}
\int_0^1\tfrac{\exp\bigl[u\a_1G_1(x,\nu)+u^2\a_2 G_2(x,\nu)\bigr]-1}{1-x}\,dx\\
=\a_1 u\bigl(-2\zeta(2)\log\nu + 2\zeta(3)\bigr)+u^2\biggl(\tfrac{\a_1^2}{2}\bigl(8\zeta(3)\log^2\nu-24\zeta(4)\log\nu\bigr)\\
+\a_2\bigl(-3\zeta(2)\log^2\nu+6\zeta(3)\log\nu-6\zeta(4)\bigr)\biggr)+O\bigl(u^3\log^3\nu\bigr).
\end{multline*}
Upon expansion $e^u=1+u+u^2/2+O(u^3)$, the bottom RHS in \eqref{1.23new} then becomes
\begin{multline}\label{1.24new}
1+u\Bigl(1+\a_1\tfrac{-2\zeta(2)\log\nu+2\zeta(3)}{h_{\nu-1}}\Bigr)+u^2\biggl(\tfrac{\tfrac{\a_1^2}{2}\bigl(8\zeta(3)\log^2\nu-24\zeta(4)\log\nu\bigr)}{h_{\nu-1}}\\
+\tfrac{a_2\bigl(-3\zeta(2)\log^2\nu+6\zeta(3)\log\nu-6\zeta(4)\bigr)}{h_{\nu-1}}+\a_1\tfrac{-2\zeta(2)\log\nu+2\zeta(3)}
{h_{\nu-1}}\biggr)+O\bigl(u^3\log^2\nu\bigr),\\
=1+u\Bigl(1+\a_1\tfrac{-2\zeta(2)\log\nu+2\zeta(3)}{h_{\nu-1}}\Bigr)+O\bigl(u^3\log^2\nu\bigr),
\end{multline}
if, leaving $\a_1=\a_1(\nu)>0$ to be determined shortly, we select $\a_2=\a_2(\nu) $ to make the coefficient by $u^2$
equal to zero. Looking closer at the coefficient by $u^2$, we see that this 
\[
\a_2=\tfrac{4\a_1^2\zeta(3)}{3\zeta(2)}+O\bigl(\log^{-1}\nu\bigr).
\]
The rest is short. Pick $\a_1=(2\zeta(2))^{-1} (1+u^b)$, $b<2\delta/3$. Since $u=\Theta(\log^{-3/2}n)$, the bottom expression in 
\eqref{1.24new} becomes
\begin{multline*}
1+u\bigl(-u^b+O(\log^{-1}\nu)\bigr)+O(u^3\log^2\nu)\\
=1-u^{b+1}\bigl(1+O(u^{-b}\log^{-1}\nu)+O(u^{2-b}\log^2n)\bigr)\\
=1-u^{b+1}\bigl[1+O((\log n)^{-\delta+3b/2})+O((\log n)^{-1+3b/2})\bigr]<1,
\end{multline*}
since $b<2\delta/3$. So, it follows from \eqref{1.23new} that
\[
\tfrac{e^u}{h_{\nu-1}}\sum_{k=1}^{\nu-1}\tfrac{g_k(u)}{\nu-k}<g_{\nu}(u),\quad \nu\in [\nu_n, n].
\]
Combining this recursive inequality with \eqref{1.1new}, we conclude that 
\[
\limsup_{n\to\infty}\max_{\nu\in [\nu_n,n]}\tfrac{f_{\nu}(u)}{g_{\nu}(u)}\le 1.
\]
Now,
\begin{multline*}
g_{\nu}(u)=\exp\bigl(u\a_1\log^2\nu+u^2\a_2\log^3\nu\bigr)\\
=\exp\Bigl[u\bigl((2\zeta(2))^{-1}(1+u^b)\bigr)\log^2\nu +u^2\bigl(\tfrac{\zeta(3)}{3\zeta^3(2)}+o(1))\log^3\nu\Bigr]\\
=\exp\Bigl[u(2\zeta(2))^{-1}\log^2\nu +u^2\tfrac{\zeta(3)}{3\zeta^3(2)}\log^3\nu+o(1) +O\bigl(u^{b+1}\log^2\nu\bigr)\Bigr]\\
=(1+o(1))\exp\Bigl[u(2\zeta(2))^{-1}\log^2\nu +u^2\tfrac{\zeta(3)}{3\zeta^3(2)}\log^3\nu\Bigr],
\end{multline*}
if we select $b>1/3$. Since $b<2\delta/3$, a desired $b$ exists provided that $\delta>1/2$, the constraint compatible with the initial restriction $\delta<3/4$.  We conclude that for $\delta\in (1/2,3/4)$
\[
\limsup_{n\to\infty}\max_{\nu\in [\nu_n,n]}f_{\nu}(u)\exp\Bigl[-u(2\zeta(2))^{-1}\log^2\nu -u^2\tfrac{\zeta(3)}{3\zeta^3(2)}\log^3\nu\Bigr]\le 1.
\]
Likewise, picking $\a_1=(2\zeta(2))^{-1} (1-u^b)$, $b<2\delta/3$, we obtain
\[
\liminf_{n\to\infty}\min_{\nu\in [\nu_n,n]}f_{\nu}(u)\exp\Bigl[-u(2\zeta(2))^{-1}\log^2\nu -u^2\tfrac{\zeta(3)}{3\zeta^3(2)}\log^3\nu\Bigr]\ge 1.
\]
This verifies \eqref{1new}, as required.
\end{proof}

\subsection{How soon do the species part their ways?} 
\label{sec:pruned}
Recall from section \ref{sec:outline} the notion of  {\em pruned spanning tree} on $t$ random leaves within the tree model on $n$ leaves.
Write $S_{n,t}$ for the edge height of the first branchpoint in the pruned tree.
In other words, the number of edges from the root to the vertex after which the $t$ sampled leaves are first split into some $(k,t-k)$ leaf subsets.
Conditioned on the size $k$ of the left subtree at the root of the tree with $n$ leaves, the probability that the $t$ sampled leaves are all in this left subtree is $\tfrac{(k)_t}{(n)_t}$. Therefore, since $q(n,k)=\tfrac{n}{2h_{n-1}k(n-k)}$, we obtain the recursion
\begin{equation}\label{4.052-}
\E[S_{n,t}]=1+\tfrac{1}{h_{n-1}}\sum_{k=1}^{n-1}\tfrac{(n/k)\E[S_{k,t}]}{n-k}\tfrac{(k)_{t}}{(n)_{t}},\quad n\ge t\ge 2,
\end{equation}
$(\E[S_{k,1}]=0)$, or, introducing $\Phi_{n,t}=(n-1)_{t-1}\E[S_{n,t}]$,
\begin{equation}\label{4.052}
\Phi_{n,t}=(n-1)_{t-1}+\tfrac{1}{h_{n-1}}\sum_{k=1}^{n-1}\tfrac{\Phi_{k,t}}{n-k}.
\end{equation}

\begin{proposition}\label{thm11}
\[
\E[S_{n,t}]=\tfrac{\log n}{h_{t-1}}+O(1) \mbox{ as } n \to \infty.
\]
\end{proposition}
\begin{proof} 
Given $\a>0$, define 
\[
U_{\nu,t}=\Phi_{\nu,t}-\a\nu^{t-1}\log  \nu.
\]
Then, by \eqref{4.052}, we have
\begin{multline}\label{4.053}
U_{\nu,t}\!=\!(\nu-1)_{t-1}+\tfrac{1}{h_{\nu-1}}\sum_{k=1}^{\nu-1}\tfrac{U_{k,t}}{\nu-k}
+\a\biggl(\tfrac{1}{h_{\nu-1}}\sum_{k=1}^{\nu-1}\tfrac{k^{t-1}\log k}{\nu-k} -\nu^{t-1}\log \nu\biggr),\\
\end{multline}
and the coefficient by $\a$ equals
\begin{multline*}
\tfrac{\nu^{t-1}}{h_{\nu-1}}\sum_{k=1}^{\nu-1}\tfrac{(k/\nu)^{t-1}[\log\nu+\log (k/\nu)]}{\nu-k}-\nu^{t-1}\log\nu\\
=\tfrac{\nu^{t-1}}{h_{\nu-1}}\biggl(\log\nu\sum_{k=1}^{\nu-1}\tfrac{(k/\nu)^{t-1}-1}{\nu-k}+\log\nu\sum_{k=1}^{\nu-1}\tfrac{1}{\nu-k}
+\sum_{k=1}^{\nu-1}\tfrac{(k/\nu)^{t-1}\log(k/\nu)}{\nu-k}\biggr)-\nu^{t-1}\log\nu\\
=\tfrac{\nu^{t-1}}{h_{\nu-1}}\biggl(\log\nu\int_0^1\tfrac{x^{t-1}-1}{1-x}\,dx+h_{\nu-1}\log\nu+O(1)\biggr)-\nu^{t-1}\log\nu\\
=-\tfrac{\nu^{t-1}\log\nu}{h_{\nu-1}}h_{t-1}+O\bigl(\nu^{t-1}\log^{-1}\nu\bigr).
\end{multline*}
So, the equation \eqref{4.053} becomes
\begin{multline}\label{4.0531}
U_{\nu,t}\!=\!(\nu-1)_{t-1}+\a\Bigl(\!-\tfrac{\nu^{t-1}\log\nu}{h_{\nu-1}}h_{t-1}+O\bigl(\nu^{t-1}\log^{-1}\nu\bigr)\Bigr)+\tfrac{1}{h_{\nu-1}}\sum_{k=1}^{\nu-1}\tfrac{U_{k,t}}{\nu-k}\\
=O\bigl(\nu^{t-1}\log^{-1}\nu\bigr)+\tfrac{1}{h_{\nu-1}}\sum_{k=1}^{\nu-1}\tfrac{U_{k,t}}{\nu-k}
\end{multline}
if we choose $\a=\tfrac{1}{h_{t-1}}$. Consequently, for some constant $\be$,
\[
|U_{\nu,t}|\le \be\nu^{t-1}\log^{-1}\nu+\tfrac{1}{h_{\nu-1}}\sum_{k=1}^{\nu-1}\tfrac{|U_{k,t}|}{\nu-k}.
\]
For a constant $B$, to be chosen shortly, we have
\begin{multline*}
\be\nu^{t-1}\log^{-1}\nu+\tfrac{1}{h_{\nu-1}}\sum_{k=1}^{\nu-1}\tfrac{B k^{t-1}}{\nu-k}
=\be\nu^{t-1}\log^{-1}\nu+\tfrac{B\nu^{t-1}}{h_{\nu-1}}\sum_{k=1}^{\nu-1}\tfrac{(k/\nu)^{t-1}}{\nu-k}\\
=\be\nu^{t-1}\log^{-1}\nu+\tfrac{B\nu^{t-1}}{h_{\nu-1}}\biggl(h_{\nu-1}+\int_0^1\tfrac{x^{t-1}-1}{1-x}\,dx+O(\nu^{-1})\biggr)\\
=\be\nu^{t-1}\log^{-1}\nu+\tfrac{B\nu^{t-1}}{h_{\nu-1}}\bigl(h_{\nu-1}-h_{t-1}+O(\nu^{-1})\bigr)
<B\nu^{t-1},\\
\end{multline*}
provided that
\[
\be\log^{-1}\nu-B\Bigl(\tfrac{h_{t-1}}{h_{\nu-1}}+O(\nu^{-1})\Bigr)<0.
\]
And this inequality holds for all $\nu\ge 2$, if we choose $B$ sufficiently large. It follows, by induction on $\nu$, that $|U_{\nu,t}|\le B\nu^{t-1}$. Consequently
\[
\Phi_{\nu,t}=\a\nu^{t-1}\log \nu+O(\nu^{t-1}),
\]
so that
\[
\E[S_{\nu,t}]=\tfrac{\Phi_{\nu,t}}{(\nu-1)_{t-1}}=\a\log\nu+O(1),\quad \a=\tfrac{1}{h_{t-1}}.
\]
\end{proof}

Within the same notion of  {\em pruned spanning tree} on $t$ random leaves within the tree model on $n$ leaves, 
a more complicated statistic is the edge-length of the pruned tree, which we denote as $S^*_{n,t}$.
To derive the counterpart of \eqref{4.052-}, notice that the total number of ways to partition the set
$[n]\setminus [t]$ into two trees, the left one of cardinality $k$, with $t_1\le t$ vertices from $[t]$ and the right one of cardinality $n-k$, with $t_2=t-t_1$ remaining vertices from $[t]$, equals $\binom{n-t }{k-t_1}$. 
Defining $S^*_{n,0}=0$, $S^*_{n,1}=0$,  $\forall\,n\ge 0$, we have the recursion: for $n\ge t\ge 2$, 
\begin{multline*}
\E[S^*_{n,t}]=1+\sum_{k=1}^{n-1}\tfrac{n}{2h_{n-1} k(n-k)}\cdot\binom{n}{k}^{-1}\\
\times\sum_{t_1\le t}\binom{n-t}{k-t_1}\Bigl(\E[S^*_{k,t_1}]+\E[S^*_{n-k,t_2}]\Bigr)\\
=1+\sum_{k=1}^{n-1}\tfrac{n}{2h_{n-1} k(n-k)}\sum_{t_1\le t}\tfrac{(k)_{t_1} (n-k)_{t_2}}{(n)_t}\Bigl(\E[S^*_{k,t_1}]+\E[S^*_{n-k,t_2}]\Bigr)\\
=1+\tfrac{1}{h_{n-1}}\sum_{k=2}^{n-1}\sum_{t_1=2}^t\tfrac{(k-1)_{t_1-1}(n-k)_{t_2}}{(n-1)_{t-1}(n-k)}\,\E[S^*_{k,t_1}].
\end{multline*}
Therefore, with $\Psi_{n,t}:=(n-1)_{t-1}\E[S^*_{n,t}]$, so that $\Psi_{n,0}=\Psi_{n,1}=0$, $\Psi_{n,t}=0$ for
$n<t$, we obtain
\begin{equation}\label{4.054}
\Psi_{n,t}=(n-1)_{t-1}+\tfrac{1}{h_{n-1}} \sum_{t_1=2}^{t}\sum_{k=2}^{n-1}\tfrac{(n-k)_{t_2}}{n-k}\Psi_{k,t_1}, \quad n\ge t\ge 2.
\end{equation}
This equation is similar to \eqref{4.052}. Because of the new factor $(n-k)_{t_2}$, we will use 
\begin{equation}\label{Stir}
(a)_b=\sum_{j=1}^b s(b,j)a^j,
\end{equation}
where $s(b,j)$ is the signed Stirling number of the first kind, so that $|s(b,j)|$ is the total number of permutations of $[b]$ with $j$ cycles.

We now repeat the statement of Theorem \ref{thmH}.
\begin{proposition}\label{thm12}
\[
\E[\mathcal S_{n,t}^*]=\a(t)\log n+O(1),\quad  \a(t)=\biggl(h_{t-1}-\sum_{t_1+t_2=t}\tfrac{(t_1-1)!(t_2-1)!}{(t-1)!}\biggr)^{-1}.
\]
\end{proposition}
\begin{proof} The argument is guided by the proof of Proposition \ref{thm11}. Given $\a>0$, define 
\[
V_{\nu,t}=\Psi_{\nu,t}-\a\nu^{t-1}\log \nu,\quad \nu\ge t\ge 2.
\]
By \eqref{4.054}, we have
\begin{multline}\label{4.055}
V_{\nu,t}\!=\!(\nu-1)_{t-1}+\tfrac{1}{h_{n-1}} \sum_{t_1=2}^{t}\sum_{k=2}^{\nu-1}\tfrac{(\nu-k)_{t_2}}{\nu-k}V_{k,t_1}\\
+\a\biggl(\tfrac{1}{h_{\nu-1}}\sum_{t_1=2}^t\sum_{k=2}^{\nu-1}\tfrac{(\nu-k)_{t_2}}{\nu-k}\,k^{t_1-1}\log k -\nu^{t-1}\log \nu\biggr).
\end{multline}
Consider the factor by $\a$. By \eqref{Stir},
\begin{align*}
&\sum_{k=2}^{\nu-1}\tfrac{(\nu-k)_{t_2}}{\nu-k} k^{t_1-1}\log k=\sum_{j=0}^{t_2}s(t_2,j)\Sigma(\nu,t_1,j),\\
&\qquad\qquad\Sigma(\nu,t_1,j):=\sum_{k=2}^{\nu-1}(\nu-k)^{j-1}k^{t_1-1}\log k.
\end{align*}
Recalling that $t_1>1$, we write
\begin{multline*}
\Sigma(\nu,t_1,0)=\sum_{k=2}^{\nu-1}\tfrac{k^{t_1-1}\log k}{\nu-k}=\sum_{k=2}^{\nu-1}\tfrac{k^{t_1-1}\bigl(\log\nu+\log(k/\nu)\bigr)}{\nu-k}
\\
=(\log\nu)\biggl(\nu^{t_1-1}h_{\nu-1}+\sum_{k=2}^{\nu-1}\tfrac{k^{t_1-1}-\nu^{t_1-1}}{\nu-k}\biggr)
+\sum_{k=2}^{\nu-1}\tfrac{k^{t_1-1}\log(k/\nu)}{\nu-k},
\end{multline*}
and
\begin{multline*}
\sum_{k=2}^{\nu-1}\tfrac{k^{t_1-1}-\nu^{t_1-1}}{\nu-k}=\nu^{t_1-1}\biggl(\int_0^1\tfrac{x^{t_1-1}-1}{1-x}\,dx+O(\nu^{-1})\biggr)\\
=\nu^{t_1-1}\biggl(-\int_0^1\sum_{s=0}^{t_1-2}x^s\,dx+O(\nu^{-1})\biggr)\\
=-\nu^{t_1-1}h_{t_1-1}+O(\nu^{t_1-2}),
\end{multline*}
while it is easy to see that $\sum_{k=2}^{\nu-1}\tfrac{k^{t_1-1}\log(k/\nu)}{\nu-k}$ is of order 
$\nu^{t_1-1}\int_0^1\tfrac{x^{t_1-1}\log x}{1-x}\,dx=O(\nu^{t_1-1})$.
Therefore
\begin{equation}\label{>0,0}
\Sigma(\nu,t_1,0)=(h_{\nu-1}-h_{t_1-1})\nu^{t_1-1}\log\nu +O(\nu^{t_1-1}).
\end{equation}
Suppose that $j>0$. Then
\begin{multline}\label{>0,>0}
\Sigma(\nu,t_1,j)=\nu^{t_1+j-1}\biggl(\nu^{-1}\sum_{k=1}^{\nu-1}\bigl(1-k/\nu\bigr)^{j-1}\bigl(k/\nu\bigr)^{t_1-1}\bigl[\log\nu+\log(k/\nu)\bigr]\biggr)\\
=\nu^{t_1+j-1}\biggl[(\log\nu)\int_0^1(1-x)^{j-1}x^{t_1-1}\,dx\\
\qquad+\int_0^1(1-x)^{j-1}x^{t_1-1}(\log x)\,dx+O(\nu^{-1}\log\nu)\biggr]\\
=\tfrac{(j-1)! (t_1-1)!}{(t_1+j-1)!}\cdot\nu^{t_1+j-1}\log\nu +O(\nu^{t_1+j-2}\log\nu),
\end{multline}
and $t_1+j-1\le t_1+t_2-1=t-1$. Combining \eqref{>0,0} and \eqref{>0,>0}, and using $s(b,b)=1$, $s(b,0)=0$ for $b>0$, we have
\begin{multline*}
\sum_{k=2}^{\nu-1}\tfrac{(\nu-k)_{t_2}}{\nu-k} k^{t_1-1}\log k\\
=(h_{\nu-1}-h_{t_1-1})\nu^{t_1-1}\log\nu +\tfrac{(t_2-1)!(t_1-1)!}{(t-1)!}\nu^{t-1}\log \nu+O(\nu^{t-2}\log\nu).\\
\end{multline*}
So, the factor by $\a$ in \eqref{4.055} is 
\begin{multline*}
\tfrac{\nu^{t-1}\log\nu}{h_{\nu-1}}\biggl(h_{\nu-1}-h_{t-1}+\sum_{t_1=1}^t\tfrac{(t_2-1)!(t_1-1)!}{(t-1)!}+O(\nu^{-1})\biggr)-\nu^{t-1}\log\nu\\
=\tfrac{\nu^{t-1}\log\nu}{h_{\nu-1}}\biggl(-h_{t-1}+\sum_{t_1=1}^t\tfrac{(t_2-1)!(t_1-1)!}{(t-1)!}+O(\nu^{-1})\biggr).
\end{multline*}
Consequently the equation \eqref{4.055} becomes
\begin{multline*}
V_{\nu,t}\!=\!(\nu-1)_{t-1}+\a\tfrac{\nu^{t-1}\log\nu}{h_{\nu-1}}\biggl(-h_{t-1}+\sum_{t_1=1}^t\tfrac{(t_2-1)!(t_1-1)!}{(t-1)!}+O(\nu^{-1})\biggr)\\
+\tfrac{1}{h_{n-1}} \sum_{t_1=2}^{t}\sum_{k=2}^{\nu-1}\tfrac{(\nu-k)_{t_2}}{\nu-k}V_{k,t_1}\\
=O\bigl(\nu^{t-1}\log^{-1}\nu\bigr)+\tfrac{1}{h_{n-1}} \sum_{t_1=2}^{t}\sum_{k=2}^{\nu-1}\tfrac{(\nu-k)_{t_2}}{\nu-k}V_{k,t_1},
\end{multline*}
if we select
\[
\a=\biggl(h_{t-1}-\sum_{t_1+t_2=t}\tfrac{(t_1-1)!(t_2-1)!}{(t-1)!}\biggr)^{-1}.
\]
We omit the rest of the proof since it runs just like the final part of the proof of Proposition \ref{thm11}. 
\end{proof}

\subsection{Counting the subtrees by the number of their leaves.} 
\label{sec:countingsubtrees} Since the tree with $n$ leaves has $2n-1$ vertices, there are exactly $2n-1$ subtrees, with the number of leaves ranging, with possible gaps, from $1$ to $n$. Let $X_n(t)$ be the number of subtrees with $t$ leaves; so $X_n(1)=n$, $X_n(n)=1$, and $X_n(t)=0$ for $t>n$. Now, $\sum_{t\ge 1}X_n(t)=2n-1$, so $\{u_n(t)\}_{t\ge 1}:=\bigl\{\tfrac{E[X_n(t)]}{2n-1}\bigl\}_{t\ge 1}$ is the probability distribution of the number of leaves in the uniformly random subtree, i.e. the subtree rooted at the uniformly random vertex of the whole tree. Furthermore
\begin{equation}\label{sub1} 
E[X_n(t)]=\tfrac{n}{2h_{n-1}}\sum_{j=1}^{n-1}\tfrac{E[X_j(t)]+E[X_{n-j}(t)]}{j(n-j)}=\tfrac{n}{h_{n-1}}\sum_{j=1}^{n-1}\tfrac{E[X_j(t)]}{j(n-j)}.
\end{equation}
So, with $\xi_n(t):=\tfrac{E[X_n(t)]}{n}$, and $h_k:=\sum_{j=1}^k\tfrac{1}{j}$, we have
\begin{equation}\label{sub1.01}
 \xi_n(t)=\tfrac{1}{h_{n-1}}\sum_{j=t}^{n-1}\tfrac{\xi_j(t)}{n-j},\quad n\ge t+1,\,\,\bigl(\xi_t(t)=\tfrac{1}{t}\bigr).
\end{equation}
and clearly  $u_n(t)=\tfrac{\xi_n(t)}{2-n^{-1}}$.

\begin{theorem}\label{thmt>2} {\bf (i)\/} $\xi_n(t)\in\bigl[\tfrac{1}{t^2},\tfrac{1}{th_t}\bigr],\,\,\tfrac{1}{t}\le \sum_{\tau\ge t}\xi_n(\tau)\le \tfrac{2}{t}$, the last bound implying that the sequence of
distributions $\{u_n(t)\}_{t\ge 1}$ is tight. \\
{\bf (ii)\/} For $q\in (0,1)$, $F_n(q):=\sum_{t\ge 1}q^t\xi_n(t)$ decreases with $n$.  Consequently, the sequence of
distributions $\{u_n(t)\}_{t\ge 1}$ converges to a proper distribution $\{u(t)\}_{t\ge 1}$. \\
{\bf (iii)\/} However,
the expected size of the uniformly random subtree is asymptotic to $\tfrac{3}{2\pi^2}\log^2 n$.
\end{theorem}

\noindent 
We conjecture that (ii) can be improved to the stronger assertion that $\xi_n(t)$ is decreasing with $n$, for each $t$.
We are grateful to Huseyin Acan \cite{Aca} 
for numerically verifying this for $n$ and $t$ below $1000$. 

\begin{proof} 
{\bf (i)\/} Let us show that $\xi_n(t)\ge \tfrac{1}{t^2}$ for $n\ge t>1$. By \eqref{sub1}, we have $\xi_t(t)=\tfrac{1}{t}$ and $\xi_{t+1}(t)=\tfrac{1}{th_t}$, both above $\tfrac{1}{t^2}$. Suppose that $n\ge t+1$ is such that $\xi_j(t)\ge \tfrac{1}{t^2}$ for all $j\in [t,n]$. This is true for $n=t+1$. For $n> t+1$, 
\begin{multline*}
\xi_n(t)\ge\tfrac{\xi_t(t)}{h_{n-1}(n-t)}+\tfrac{1}{t^2h_{n-1}}\sum_{j=t+1}^{n-1}\tfrac{a}{n-j}
= \tfrac{1}{h_{n-1}(n-t)t}+\tfrac{h_{n-1-t}}{t^2h_{n-1}}\\
=\tfrac{1}{t^2}+ \tfrac{1}{h_{n-1}(n-t)t}+\tfrac{h_{n-1-t}-h_{n-1}}{t^2 h_{n-1}}\\
\ge \tfrac{1}{t^2}+\tfrac{1}{h_{n-1}(n-t)t}-\tfrac{1}{t^2 h_{n-1}}\cdot \tfrac{t}{n-t} =\tfrac{1}{t^2},
\end{multline*}
which completes the induction step. The proof of $\xi_n(t)\le\tfrac{1}{th_t}$ is similarly reduced to showing that 
$\tfrac{(n-1)h_t}{(n-t)th_{n-1}}\le 1$ for $n>t+1$.
This is so, as the fraction is at most $\tfrac{h_t}{h_{t+1}}\cdot\tfrac{t+1}{2t}$. 

Let us prove that $\tfrac{1}{t}\le \sum_{\tau\ge t}\xi_n(\tau)\le \tfrac{2}{t}$. Introduce
$Y_n(t)=\sum_{\tau\ge t}X_n(\tau)$, the total number of subtrees with at least $t$ leaves, and $\eta_n(t):=\tfrac{\E[Y_n(t)]}{n}=\sum_{\tau\ge t}\xi_n(\tau)$; so $\eta_n(1)=\tfrac{2n-1}{n}$, and $\eta_n(n)=\tfrac{1}{n}$.  Analogously to \eqref{sub1}, we have
\[
\eta_n(t)=\tfrac{1}{h_{n-1}}\sum_{j=t}^{n-1}\tfrac{\eta_j(t)}{n-j},\quad n\ge t+1.
\]
We need to show that  $\eta_n(t)\le\tfrac{2}{t}$ for all $n\ge t$. It suffices to consider $n>t>1$. Suppose that for some $n\ge t$ and all 
$j\in [t,n]$ we have $\eta_j(t)\le \tfrac{2}{t}$. This is definitely true for $n=t$. Then
\begin{multline*}
\eta_{n+1}(t)=\tfrac{1}{h_n}\sum_{j=t}^n\tfrac{\eta_j(t)}{n+1-j}\le\tfrac{2}{th_n}\sum_{j=t}^n\tfrac{1}{n+1-j}
=\tfrac{2h_{n+1-t}}{th_n}\le\tfrac{2}{t},\\
\end{multline*}
which competes the inductive proof of $\eta_n(t)\le \tfrac{2}{t}$. \\
 {\bf (ii)\/} For $n\ge 2$, we have
\begin{multline*}
F_n(q)=\sum_{t\ge 1}q^t\xi_n(t)=\sum_{t\ge 1}\tfrac{q^t}{h_{n-1}}\sum_{j=t}^{n-1}\tfrac{\xi_j(t)}{n-j}\\
=\tfrac{1}{h_{n-1}}
\sum_{j=1}^{n-1}\tfrac{1}{n-j}\sum_{t=1}^jq^t\xi_j(t)=\tfrac{1}{h_{n-1}}\sum_{j=1}^{n-1}\tfrac{F_j(q)}{n-j}.\\
\end{multline*}
Therefore
\begin{multline*}
F_{n+1}(q)=\tfrac{1}{h_n}\sum_{j=1}^n\tfrac{F_j(q)}{n+1-j}=\tfrac{1}{h_n}\biggl(\tfrac{F_1(q)}{n}+\sum_{j=2}^n\tfrac{F_j(q)}{n+1-j}\biggr)\\
\le \tfrac{1}{h_n}\biggl(\tfrac{F_1(q)}{n}+\sum_{j=1}^{n-1}\tfrac{F_j(q)}{n-j}\biggr)
\le\tfrac{1}{h_n}\bigl(\tfrac{q}{n}+h_{n-1}F_n(q)\bigr)\le F_n(q),
\end{multline*}
since $F_n(q)\ge q\xi_n(t)=q$, and $h_n-h_{n-1}=\tfrac{1}{n}$. 
Therefore, for each $q\in (0,1)$ there exists $F(q)=\lim_{n\to\infty}\sum_{t\ge1 }q^t\xi_n(t)$, implying existence of
 $\lim_{n\to\infty}\sum_{t\ge 1}q^t u_n(t)=2F(q)$.
For any weakly convergent subsequence of the distributions $\{u_{n_{i}}(t)\}_{t\ge 1}$, for each $x$ in the unit disc  there exists 
$\lim_{n_i\to\infty}\sum_{t\ge 1}x^t u_{n_{i}}(t)$, dependent on the subsequence, which is analytic within the disc. All these limits coincide for $x\in (0,1)$, whence they coincide for all $x$ within the disc, whence on the whole disc. Since the characteristic function determines the distribution uniquely, we see that the whole sequence of the distributions converges 
to a proper distribution.\\
{\bf (iii)\/} $Z_n:=\sum_{t\ge 1}tX_n(t)$ is the total number of the leaves, each leaf counted as many times as the
number of the subtrees rooted at the vertices along the path from the root to the leaf, which is distributed as $1$ plus
$L_n$, the edge-length of the path to the random leaf. 
Therefore $\tfrac{\E[Z_n]}{2n-1}=\tfrac{n}{2n-1}\bigl(1+\E[L_n]\bigr)$, and it remains to use Proposition \ref{2new}.
\end{proof}

\section{Other methods}
\label{sec:3}
Our results here demonstrate that the analysis of recursions method is very effective at deriving sharp asymptotics for the questions 
addressed here.
However, there are many other aspects of the model that could be studied, 
and a wide variety of familiar general modern probabilistic techniques that could be applied. 
We indicate such possibilities briefly below -- see the preprint  \cite{Ald2} for a more comprehensive account.

 It is intuitively clear that 
for our tree model (call it $\TT_n$, say) there should be two $n \to \infty$ limit structures:\\
(a) A scaling limit process, which is a fragmentation of the unit interval via some sigma-finite splitting measure.\\
(b) A fringe process, which is the local weak limit relative to a random leaf, describable as some 
marked branching process. 
This starts with an explicit description of the limit distribution $\{u(t)\}_{t\ge 1}$ in Theorem \ref{thmt>2}. \\
Less intuitive is a piece of structure theory.
In our discrete-time model there is no simple connection between $\TT_n$ and $\TT_{n+1}$.
In the continuous-time model with our chosen rates $h_{n-1}$ 
(and this is the reason for that particular choice),  \cite{Ald2} shows there is a non-obvious consistency property under a ``delete a random leaf and prune" operation. 
This enables  an inductive construction of a process $(\TT_n, n = 2,3,4, \ldots)$, which in turn suggests the possibility of a.s. limit theorems.

Readers may notice that the issues above are somewhat analogous to those arising in the theory
surrounding the Brownian continuum random tree (CRT) as a limit of certain other random tree models
\cite{crt2,evans}. 
However, in contrast to the CRT setting, the two limit processes above would not capture the asymptotics of the quantities studied in this article.

{\bf Acknowledgment.\/} We thank Svante Janson for catching an error in an earlier version. We are grateful to Huseyin Acan for the numerics below Theorem \ref{thmt>2}. 
We thank the referees for the time and effort to repeatedly read the paper and to provide expert critical advice. We thank the Editors for the efficient reviewing process.

\end{document}